\documentclass{article}
\usepackage{arxiv}
\usepackage[utf8]{inputenc} 
\usepackage[T1]{fontenc}    
\usepackage[hyperindex,colorlinks,citecolor=blue]{hyperref} 
\usepackage{url}            
\usepackage{booktabs}       
\usepackage{amsfonts}       
\usepackage{nicefrac}       
\usepackage{microtype}      
\usepackage[dvipsnames]{xcolor}
\usepackage[toc, page, header]{appendix}
\usepackage{titletoc}

\usepackage{times}

\usepackage{soul}
\usepackage{url}
\usepackage[hyperindex,colorlinks,citecolor=blue]{hyperref} 
\usepackage{graphicx}
\usepackage{amsmath, amssymb}
\usepackage[flushleft]{threeparttable}
\usepackage{makecell}
\usepackage{thm-restate}
\usepackage{enumitem}
\usepackage{float}
\usepackage{subcaption}
\usepackage{multirow}
\usepackage{xspace}
\usepackage{resizegather}
\usepackage{pifont}
\newcommand{\cmark}{\ding{51}}%
\newcommand{\xmark}{\ding{55}}%
\usepackage[font=small]{caption}

\usepackage{verbatim}

\usepackage{amsmath,amsthm,amssymb}
\theoremstyle{plain}
\newtheorem{theorem}{Theorem}[section]
\newtheorem{lemma}[theorem]{Lemma}

\newtheorem{proposition}[theorem]{Proposition}
\newtheorem{assumption}{Assumption}
\newtheorem{claim}[theorem]{Claim}

\theoremstyle{definition}

\theoremstyle{remark}
\newtheorem{remark}{Remark}

\providecommand{\R}{\mathbb{R}} 

\providecommand{\bb}{\mathbf{b}}
\providecommand{\cc}{\mathbf{c}}

\providecommand{\xx}{\mathbf{x}}
\providecommand{\yy}{\mathbf{y}}
\providecommand{\zz}{\mathbf{z}}
\providecommand{\unit}{\mathbf{1}}
\providecommand{\zero}{\mathbf{0}}
\providecommand{\mA}{\mathbf{A}}
\providecommand{\mB}{\mathbf{B}}
\providecommand{\mC}{\mathbf{C}}

\providecommand{\mE}{\mathbf{E}}

\providecommand{\mG}{\mathbf{G}}

\providecommand{\mI}{\mathbf{I}}
\providecommand{\mJ}{\mathbf{J}}

\providecommand{\mW}{\mathbf{W}}
\providecommand{\mX}{\mathbf{X}}
\providecommand{\mY}{\mathbf{Y}}
\providecommand{\mZ}{\mathbf{Z}}

\providecommand{\cC}{\mathcal{C}}
\providecommand{\cD}{\mathcal{D}}
\providecommand{\cE}{\mathcal{E}}

\providecommand{\cH}{\mathcal{H}}

\providecommand{\cN}{\mathcal{N}}
\providecommand{\cO}{\mathcal{O}}

\providecommand{\cV}{\mathcal{V}}

\newenvironment{talign*}
{\csname align*\endcsname}
{\endalign}
\newcommand{\iid}{i.i.d.\xspace}

\usepackage{todonotes}

\usepackage{algorithm}
\usepackage{algpseudocode}
\usepackage{footnote}
\usepackage{tablefootnote}
\usepackage{wrapfig}
\usepackage{caption}
\usepackage[numbers,sort&compress]{natbib}
\bibpunct[, ]{[}{]}{,}{n}{,}{,}
\renewcommand\bibfont{\fontsize{10}{12}\selectfont}

\begin{document}
\setlength{\abovedisplayskip}{3pt}
\setlength{\belowdisplayskip}{3pt}
\title{Decentralized Gradient Tracking with Local Steps}

\author{
  Yue Liu \\
  University of Toronto\\
  Canada \\
   \And
  Tao Lin  \\
  Westlake University  \\
  China \\
  \And
  Anastasia Koloskova  \\
  EPFL  \\
  Switzerland \\
  \And
  Sebastian U. Stich  \\
  CISPA Helmholtz Center for Information Security  \\
  Germany \\
}

\maketitle

\begin{abstract}
Gradient tracking (GT) is an algorithm designed for solving decentralized optimization problems over a network (such as training a machine learning model). 
A key feature of GT is a tracking mechanism that allows to overcome data heterogeneity between nodes. 

We develop a novel decentralized tracking mechanism, $K$-GT, that enables communication-efficient local updates in GT while inheriting the data-independence property of GT. 
We prove a convergence rate for $K$-GT on smooth non-convex functions and prove that it reduces the communication overhead asymptotically by a linear factor $K$, where $K$ denotes the number of local steps.
We illustrate the robustness and effectiveness of this heterogeneity correction on convex and non-convex benchmark problems and on a non-convex neural network training task with the MNIST dataset.

\end{abstract}

\section{Introduction}


We consider distributed optimization problems,
where the objective function $f(\xx)$ on model $\xx \in \R^d$ is defined as the average of $n$ different components $\{f_1(\xx), ..., f_n(\xx)\}$, i.e., $$f(\xx)=\frac{1}{n}\sum_{i=1}^nf_i(\xx).$$  In distributed applications, different contributors (or `clients') take part in the training.
Such clients can be, for example, mobile edge devices, or computing nodes. Typically, each component $f_i(\xx)$ is only available to a single client (for instance, when $f_i(\xx)$ is defined over the training data available only locally  on the client). This makes distributed optimization problems more difficult to solve than centralized problems.

Besides convergence rate in terms of iterations, communication efficiency is one of the most important metrics in distributed algorithm design. For illustration we consider the calculation of a gradient of the global function, $\nabla f(\xx) = \frac{1}{n}\sum_i\nabla f_i(\xx)$, that forms be basis for general first-order methods. Since each client only has the ability to evaluate the local gradient $\nabla f_i(\xx)$, it is further necessary to calculate the average of these local gradients.
Centralized algorithms~\citep{fedcomm, mcmahan2017communication} realize such global aggregation by a central controller, e.g., with a parameter sever~\citep{ps}. However, this approach requires all clients to communicate with the central server simultaneously, resulting in a communication bottleneck at this hot point and a slowdown in clock time. Instead of the exact averaging, decentralized algorithms~\citep{tsitsiklis1984problems, de>cen,KoloskovaLSJ19decentralized} require only partial communication through gossip averaging and reduce communication overhead by allowing a node to communicate with fewer nodes, e.g., only its neighbors, thus avoiding having the busiest point.  How nodes are connected between each other makes up the network topology.\looseness=-1


One of the most challenging aspects in decentralized optimization is data-heterogeneity, that is when the training data is not identically and independently (non-\iid) distributed across the nodes.
Such non-\iid distributions often arise in practical applications, since, for example, training data originating from cell phones, sensors, or hospitals can have regional differences~\cite{client-sampling}.
In this case, the local empirical losses on each client are different.
This can slow down the convergence~\citep{dsgd, unified-dsgd} or even yield local overfitting (often termed \emph{client\nobreakdash-drift}) as the clients
may drift away from the global optimum in the course of the optimization process~\citep{scaffold, non-iid-benchmark, noniid, unified-dsgd}.

There are several decentralized algorithms that have been shown to mitigate heterogeneity. However, most of them are proven to converge for only strongly convex functions~\citep{extra_convex,primal-dual,diging,gt-svrg,network-svrg}, or are proven for smooth non-convex functions but have a strict constraint on network topologies~\citep[such as e.g.][]{d2_nonconvex}.  
Stochastic gradient tracking~(GT)~\citep{Lorenzo2016GT-first-paper, diging, gt, gt-nonconvex, koloskova2021improved} algorithms have been proposed to address data-heterogeneity for arbitrary networks for smooth non-convex functions. Its convergence rate only depends on the data heterogeneity at the initial point, which can be completely removed with proper initialization.
However, the clients are required to communicate with all their neighbors in the network after every single model update.
These methods are still therefore associated with high communication overheads.

In order to further reduce communication overhead within distributed training, various engineering techniques have been proposed, such as using large batch~\citep{large-mini-batch, large-mini-batch-2, largebatch}, model/gradient compression~\citep{dsgd-compression} or asynchronized communication~\citep{async-dsgd}. In this work, we focus on local updates to reduce communication frequency, which is often efficient in practice but remains challenging in the theoretical analysis~\citep{mcmahan2017communication, local-sgd1, local-sgd2,fedcomm}. However, performing a large number of local steps can
\emph{exacerbate} the client\nobreakdash-drift.
The resulting optimization difficulties can negate the communication savings~\citep{scaffold, unified-dsgd}. 
The analysis of incorporating local steps while heterogeneity independence in the decentralized optimization is still seldom investigated.

Integrating local updates into GT is non-trivial.
For instance, simply skipping communication rounds in GT (and thereby performing a number of local updates in-between)
does not work well in practice\footnote{We evaluate this variant (termed \emph{periodical GT}) below in Section~\ref{sec:experiments}, see e.g.\ 
Figure \ref{fig: mnist}.}. 
A concurrent work LU\nobreakdash-GT\citep{lu-gt} analyzed the performance of GT periodically skipping the communication but only in the deterministic setting.\footnote{This concurrent work was independently developed while we were finalizing this manuscript. We will add a more detailed comparison to the next version of this manuscript.}

As a solution, we carefully design a novel tracking mechanism that enables to combine GT with local steps.\footnote{Partial results of this paper were previously presented in YL's master thesis~\citep{thesis}}
The resulting algorithm---$K$-GT, where $K$ denotes the number of local steps---is a novel decentralized method that provides communication-efficient tracking with local updates.
We prove that the convergence of $K$-GT depends only on the data heterogeneity at the starting point and that this weak data dependence can be completely circumvented with an additional round of global communication.
As long as $K$\nobreakdash-GT uses the same initialization as GT,  $K$\nobreakdash-GT inherits the heterogeneity independence property of GT.
We prove that $K$\nobreakdash-GT~(Algorithm~\ref{algo: k-gt}) achieves asymptotically linear speed-up in terms of communication round w.r.t.\ local steps $K$ and number of clients $n$, and that it converges in $\cO\big(\frac{\sigma^2}{nK\epsilon^2}\big)$ rounds to an $\epsilon$-approximate stationary point.
The number of communication rounds is asymptotically reduced by a factor of $K$ compared to GT.
We further show that the convergence rate (including higher order terms) does not depend on
the data-heterogenity if with proper initialization, opposed as e.g.\ for decentralized stochastic gradient descent~(D\nobreakdash-SGD) without tracking. 
\looseness=-1

The outline of this paper is as follows:\ In Section 2, we give the precise formulation of the distributed optimization problem setting.  In Section 3, we introduce the algorithm design of $K$-GT and demonstrate how it helps to correct for heterogeneity. Here, our main result state its convergence rate, see Theorem~\ref{thm: k-gt}. In Section 4, we generalize the gradient tracking framework and discuss about the drawbacks of other GT alternatives that could also be stemmed from the same framework. We in addition contribute their convergence results and give a comparison to show that $K$-GT is the most communication efficiency theorectically. In Section 5, we compare the GT-variants with baseline D-SGD with numerical examples in detail.
\paragraph*{Contributions.} We summarize our main results below.
\begin{itemize}[leftmargin=12pt,nosep,itemsep=1pt]
 \item We develop a novel gradient tracking algorithm for distributed optimization and analyze its convergence properties. We prove that $K$\nobreakdash-GT enjoys heterogeneity-independent complexity estimates (with proper initialization)
  and prove that it converges asymptotically in  $\cO\big(\frac{\sigma^2}{nK\epsilon^2}\big)$ rounds, where $n$ denotes client number, $K$ the number of local steps, $\sigma^2$ the stochastic noise level and $\epsilon$ the accuracy. This improves by a factor of $K$ over the GT baseline.
 \item We provide additional theoretical insights, by studying
 (i) the convergence of the na\"ive local extension of GT, \emph{periodic GT}, explaining that it performs worse than $K$-GT when the stochastic noise is large, and
 (ii) a computationally inefficient variant, \emph{large-batch GT} that matches the iteration, but not the computation complexity of $K$-GT.
%
	\item We empirically verify the theoretical results on strongly convex and non-convex functions and explain the impact of noise, local steps and data-heterogeneity on the convergence. 
$K$-GT is robust against the data-heterogeneity while improving the communication efficiency and  improves generalization performance over baseline algorithms.
\end{itemize}

\begin{table}[tb]
\caption{A comparison under different working conditions. $\Delta \leq n$ denotes the maximum degree of the communication graph. $K$-GT is the first fully-decentralized tracking algorithm with local steps.}
\renewcommand{\arraystretch}{1.1}
	\centering
	\resizebox{\textwidth}{!}{
		\begin{tabular}{lccc}
			\toprule[1pt]
			\multirow{2}{*}{\textbf{Algorithm}} & \multicolumn{3}{c}{\textbf{Settings}}\\\cline{2-4}
			   &  \makecell[c]{Communication cost at the busiest point}& Local steps & \makecell[c]{heterogeneity-robustness\textsuperscript{a}} \\ \midrule[0.8pt]
			  \makecell[l]{SCAFFOLD~\citep{scaffold}} & $\cO(n)$& \cmark & \cmark\\\hline
			  \makecell[l]{GOSSIP\nobreakdash-PGA~\citep{gossip-pga}} & \multirow{5}{*}{$\cO(\Delta)$} & \cmark & \xmark \\\cline{1-1}\cline{3-4}
			\makecell[l]{D\nobreakdash-SGD~\citep{unified-dsgd}} &  & \cmark & \xmark \\\cline{1-1}\cline{3-4}
			\makecell[l]{GT~\citep{diging}} &  & \xmark &\cmark \\\cline{1-1}\cline{3-4}
			\makecell[l]{$D^2$~\citep{d2_nonconvex}} &  & \xmark &\cmark \\\cline{1-1}\cline{3-4}
			\makecell[l]{$K$\nobreakdash-GT {\textbf{[ours]}} }  &  & \cmark & \cmark \\
			\bottomrule[1pt]
		\end{tabular}}
\tabnote{\textsuperscript{a}The data heterogeneity does not impact the worst-case convergence rate (but might require special initialization).}		
	\label{tab: comparison}
\end{table}
\section{Problem setting}
We introduce the notation and setup in this section.
\subsection{Decentralized Optimization Problem}
We consider the optimization problems as the summation from $n$-client loss functions, \begin{equation}
	\underset{\xx \in \mathbb{R}^d}{\mbox{min}} f(\xx):=\frac{1}{n}\sum_{i=1}^n \left[f_i(x) := \mathbb{E}_{\xi_i\sim \cD_i}F_i(\xx;\xi_i) \right]\,,
	\label{eq: global function}
\end{equation}
where $n$ denotes the number of clients within the system, $\xi_i$ is a random sample from $\cD_i$ and $\cD_i$ denotes the local distribution only available on node $i\in[n]$. $\cD_i$ could be arbitrary and different among clients considering the applications. This setup  models both empirical risk minimization and the online optimization setting. 

In this work, we consider general smooth non-convex functions and bounded stochastic noise.
\begin{assumption}[Smoothness]\label{ass: smooth}Each function $f_i(x):\mathbb{R}^d\rightarrow\mathbb{R},\ \forall i\in[n]$ is differentiable and there exists a constant $L>0$ such that for each $\xx,\ \yy\in\mathbb{R}^d$, 
\begin{equation}\nonumber
	f_i(\yy)\leq f_i(\xx) + \nabla f_i(\xx)^T(\yy-\xx) + \frac{L}{2}||\xx-\yy||_2^2 \,.
\end{equation}
\label{ass:l}
\end{assumption}
\vspace{-5mm}
\begin{assumption}[Bounded variance]Each client variance is uniformly bounded, 
	\begin{equation}\nonumber
		\begin{aligned}
			  \forall i\in[n],\  \forall \xx \in \R^d,\ \mathbb{E}_{\xi \sim \cD_i}||\nabla F_i(\xx; \xi)-\nabla f_i(\xx)||_2^2\leq \sigma^2 \,.
		\end{aligned}
	\end{equation}
	\label{ass:noise}
\end{assumption}
\vspace{-10mm}

\subsection{Communication graph}The training is implemented over a decentralized network, and its topology is modelled as an undirected graph: $(\cV, \cE)$, where $\cV:=\{ 1, 2, \ldots, n \}$ is the node set and $\cE \subseteq \cV \times \cV$ is the edge set. Node~(or client) represents a computing node, and clients communicate only along the edges $e\in\cE$. We denote the adjacency matrix $\mW\in\mathbb{R}^{n\times n}$, where $w_{ij}=0$ means node $i$ and $j$ are not connected, i.e., $e_{ij}=(i,j)\notin\cE$.
\begin{assumption}[Mixing rate] Given the symmetric and doubly stochastic mixing matrix $\mW\in\mathbb{R}^{n\times n}$ of nonnegative real numbers, i.e., $\forall i,j\in[n],\ w_{ij}\geq 0$, $\textstyle \sum_{i=1}^n w_{ij} = \textstyle  \sum_{j=1}^n w_{ij} = 1,$ the consensus distance decreases linearly after averaging step, i.e. there exists a $1\geq p>0$ such that $$||\mX\mW - \bar\mX||_F^2\leq (1-p)||\mX - \bar\mX||_F^2,\ \forall \mX \in \R^{d \times n}\,.$$
\label{ass:p}
\end{assumption}\vspace{-5mm}
Note that if the commonly used network parameter $\rho:=||\mW-\frac{\unit_n\unit_n^T}{n}||$~\cite{gossip-pga} is strictly less than 1, then $1\geq p>0$ \citep[see e.g.][]{unified-dsgd, koloskova2021improved}.
The mixing rate describes the connectivity of the network. The larger value of $p$ means the communication graph is better connected. $p=1$ for a complete graph $\mW = \tfrac{1}{n}\unit\unit^T$, and $p=0$ for a disconnected graph $\mW = \mI_n$.


\subsection{Data heterogeneity and correction}
When the local distributions $\{\cD_i\}$ are identical on each client, the local functions $\{f_i(\xx)\}$ are identical to each other, i.e., $f_i(\xx)\equiv f(\xx)$. Otherwise, heterogeneous local distributions result in heterogeneous local functions. And heterogeneity is usually measured by the discrepancy between local gradients $\{\nabla f_i(\xx)\}$ and global gradient $\nabla f(\xx)$~\cite{scaffold,unified-dsgd} as follows.
\begin{assumption}[Data-heterogeneity]\label{ass: noniid} There exists constants $\zeta^2>0$ and $B\geq 1$ such that  \begin{equation}\nonumber\forall \xx\in\mathbb{R}^d,\ \frac{1}{n} \sum_{i=1}^n ||\nabla f_i(\xx)||^2\leq \bar\zeta^2+B^2||\nabla f(\xx)||^2\,,\end{equation} where both $\bar\zeta^2$ and $B^2$ represent the degree of heterogeneity within the system.
\label{ass:zeta}
\end{assumption}
The baseline Decentralized SGD~(D\nobreakdash-SGD) uses na\"{i}ve gradient w.r.t local model, the convergence of which inevitably are influenced by both $\bar\zeta^2$ and $B$~\citep{dsgd}.
\subsubsection{Notations}
Gradient tracking algorithm mainly manipulates between two variables, model iterate $\xx\in\mathbb{R}^d$ and tracking variable $\zz\in\mathbb{R}^d$. More precisely, we denote vector $\yy\in\{\xx,\ \zz\}$ as $\yy_i^{(t)+k}$ on node $i$ in local step $k$ at communication round ${t}$, and denote its average by $\bar\yy = \frac{1}{n}\sum_i\yy_i$.

The collection of vectors $\yy_i$ for all $i\in[n]$ in matrix form is denoted by a capital letter with columns $\yy_i$, i.e., $$\mY = \begin{bmatrix}\yy_1, \ldots, \yy_n\end{bmatrix}\in\mathbb{R}^{d\times n},\quad \bar\mY = \begin{bmatrix}\bar\yy, \ldots,  \bar\yy\end{bmatrix}= \tfrac{1}{n}  \mY\unit_n\unit_n^T\in\mathbb{R}^{d\times n}\,.$$
Also, we extend this matrix definition to both gradient and stochastic gradient of (\ref{eq: global function}) w.r.t model $\mX$ on sample $\xi=[\xi_1,\ \ldots,\ \xi_n]$, where $\xi_i\sim \cD_i,$
\begin{equation}\nonumber
	\begin{aligned}
		\nabla F(\mX;\xi) &=\begin{bmatrix}\nabla F_1(\xx_1;\xi_1),\ldots,\nabla F_n(\xx_n;\xi_n)\end{bmatrix}\in\mathbb{R}^{d\times n},	\\
		\nabla f(\mX)&=\mathbb{E}_{(\xi_1, \ldots, \xi_n)}\nabla F(\mX;\xi)
		=\begin{bmatrix}\nabla f_1(\xx_1), \ldots, \nabla f_n(\xx_n)\end{bmatrix}\in\mathbb{R}^{d\times n}\,.
	\end{aligned}
\end{equation}
\subsubsection{Gradient tracking}
Gradient tracking algorithm~(GT)~\cite{gt} is defined by the following update equations:
\begin{equation}\label{eq: dsgt}
	\begin{aligned}
		\mX^{(t+1)} & = (\mX^{(t)}-\eta \mZ^{(t)})\mW                               \\
		\mZ^{(t+1)} & = \mZ^{(t)}\mW +\mG^{(t+1)} -\mG^{(t)},
	\end{aligned}
\end{equation}
in matrix format. Here $\mG^{(t)} = \nabla F(\mX^{(t)}; \xi^{(t)})$ and $\eta>0$ denotes the stepsize. 

When data is heterogeneous among different nodes, $\{\nabla F_i(\xx;\xi_i),\ \forall i\}$ are different. But GT uses bias-correction to compensate heterogeneous gradient at each node. This correction is governed by the tracking variable $\mZ$ that replaces the na\"{i}ve gradient: \begin{equation}\label{eq: dsgt_corr}\mZ^{(t+1)} = \nabla F(\mX^{(t+1)}; \xi^{(t+1)}) + \underbrace{ \mZ^{(t)}\mW  -\mG^{(t)}}_{\text{correction}}\end{equation} Since the update (\ref{eq: dsgt}) simultaneously updates both the model $\mX$ and the tracking variable $\mZ$, there is no need to take extra consideration on the heterogeneous local gradient. GT is proven to converge regardless of data heterogeneity~\citep{gt}.

\section{K-GT: Gradient Sum Tracking algorithm}
In this section, we present our new decentralized stochastic algorithm $K$\nobreakdash-GT with its convergence analysis for general non-convex functions.




\vspace{-3mm}
\subsection{Algorithm}
In the $K$-GT algorithm we allow each client to perform $K \geq 1$ local steps between each communication round. To compensate to the data-heterogeneity, we use a similar correction as in  (\ref{eq: dsgt_corr}) on top of the stochastic gradient. We denote the correction as $\cc_i$ on node $i$. Then each node repeats the following updating rule, $i\in[n]$:
\begin{enumerate}
	\item Compute a local stochastic gradient $\nabla F_i(\xx_i; \xi_i)$ by sampling $\xi_i$ from distribution $\cD_i$;
	\item Update the local model $\xx_i^{(t)+k+1}= \xx_i^{(t)+k} - \eta_c \bigl(\nabla F_i(\xx_i^{(t)+k}; \xi_i^{(t)+k}) + \cc_i^{(t)}\bigr)$ using the stochastic gradients at $(t)+k$-th iteration and correction $\cc_i^{(t)}$ in $t$-th communication;
	\item Repeat step (1)-(2) $K$ times, then obtain the tracking throughout local steps, $\zz_i^{(t)}  = \frac{1}{K\eta_c} \bigl(\xx_i^{(t)} - \xx_i^{(t)+K}\bigr)$. Exchange $\{\xx_i, \cc_i\}$ with neighbors: (in matrix format):\begin{equation}\label{eq: k-gt}
	\begin{aligned}
		\mX^{(t+1)} & = \Big(\mX^{(t)} - \eta_s(\mX^{(t)}-\mX^{(t)+K})\Big)\mW	\, ,\\
		\mC^{(t+1)} & = \mC^{(t)} + \mZ^{(t)}(\mW -\mI) \,.	\\
	\end{aligned}
\end{equation}
\end{enumerate}
The complete algorithm is summarized in Algorithm~\ref{algo: k-gt}.
\vspace{-2mm}
\begin{restatable}[Gradient Sum Tracking]{proposition}{localgt}
\label{cor: local_gt}
Define $\mZ^{(t)} = \frac{1}{K\eta_c}\Big(\mX^{(t)} - \mX^{(t)+K}\Big)$ as the tracking variable during communication round~$t$. The update rule for both models $\mX^{(t)}$ and tracking variables $\mZ^{(t)}$ at communication in $K$\nobreakdash-GT can be rewritten as ($\eta=\eta_s\eta_c$):
\begin{equation}
	\begin{aligned}
		\mX^{(t+1)} & = \Big(\mX^{(t)}-K\eta\mZ^{(t)}\Big)\mW \,, \\
		\mZ^{(t+1)} & = \mZ^{(t+1)}\mW + \mG^{(t+1)} - \mG^{(t)}\,,
	\end{aligned}
	\label{eq: local_gt}
\end{equation}
where $\mG^{(t)} = \frac{1}{K} \sum_{k}\nabla F(\mX^{(t)+k};\xi^{(t)+k})$ denotes the mean update over the local steps.
\end{restatable}

The detailed proof is included in Appendix~\ref{ap: gt}.
\begin{remark}\label{remark: equivalence}
If $K=1$ in~(\ref{eq: local_gt}), $K$\nobreakdash-GT is equivalent to Gradient Tracking~\citep{gt} with $\eta=\eta_s\eta_c$.
\end{remark}
\paragraph*{$K$-GT essentially runs SGD if communication is the most sufficient.} To understand the intuition behind $K$-GT, let us consider the global average $\bar\mX$ at each iterate, which gets updated just like the standard stochastic gradient descent:
\begin{equation}	\nonumber
	\begin{aligned}	
		\bar\mX^{(t)+k+1} &= \Big(\mX^{(t)+k} - \eta_c(\nabla F(\mX^{(t)+k}; \xi^{(t)+k}) + \mC^{(t)})\Big)\tfrac{\unit\unit^T}{n}	\\
		& = \bar\mX^{(t)} - \eta_c\Big(\overline{\nabla F(\mX^{(t)+k}; \xi^{(t)+k})} + \bar\mC^{(t)}\Big) \,.
	\end{aligned}
\end{equation}
If initialized to be $\mC^{(0)} = \nabla F(\mX^{(0)}; \xi^{(0)})\bigl(\frac{\unit\unit^T}{n} - \mI\bigr)$, the average of correction satisfies 
\begin{equation}\nonumber
	\begin{aligned}
		\bar\mC^{(t+1)} & = \bar\mC^{(t)}+ \mZ^{(t)}(\mW - \mI)\tfrac{\unit\unit^T}{n} = \bar\mC^{(t)} \,,	\\
		\bar\mC^{(0)} & = \nabla F(\mX^{(0)}; \xi^{(0)})\left(\tfrac{\unit\unit^T}{n} - \mI\right)\frac{\unit\unit^T}{n}	 \equiv \zero \,.
	\end{aligned}
\end{equation}
Then the average of model iterate satisfies $$\bar\mX^{(t)+k+1} = \bar\mX^{(t)} - \eta_c\overline{\nabla F(\mX^{(t)+k}; \xi^{(t)+k})},$$ which updates model with averaged stochastic gradient. 

\paragraph*{How does this correction improves D-SGD?} We consider applying the similar analysis from~\cite{d2_nonconvex} to illustrate the effectiveness of $K$-GT.
Assume that $\mX^{(t)}$ has achieved an optimum $\mX^{\star} : = \xx^{\star}\unit^T$ with all local models equal to the optimum $\xx^{\star}$. Based on our analysis in appendix~(Lemma~\ref{lemma: correction}), the correction will be equal to $$\cc_i^{\star} :=-\nabla F_i(\xx^{\star}; \xi) + \frac{1}{n}\sum_{j}\nabla F_j(\xx_j^{\star}; \xi)\,.$$ Then the next local update for $K$-GT would be
\begin{equation}\nonumber 
	\begin{aligned}
		\mX^{(t)+1} &= \mX^{\star} -\eta_c(\nabla F(\mX^{\star}; \xi) + \mC^{\star})	\\
		& = \mX^{\star} -\eta_c\nabla F(\mX^{\star}; \xi)\tfrac{\unit\unit^T}{n} \,.
	\end{aligned}
\end{equation}
This illustration shows that for $K$-GT, the convergence when we approach a solution with only local update relies on the magnitude of $\mE||\nabla F(\mX^{\star}; \xi)\frac{\unit\unit^T}{n}||_F^2$, which is bounded by $\cO(\sigma^2)$. 

However, consider the same situation for D-SGD, 
\begin{equation}\nonumber
	\begin{aligned}
		\mX^{(t)+1} &= \mX^{\star} -\eta\nabla F(\mX^{\star}; \xi). 	\end{aligned}
\end{equation}On different nodes, $\nabla F_i(\mX^{\star}; \xi)$ deviates from each other due to data heterogeneity, and the deviation can only be characterized by $\zeta^2$ as suggested in Assumption~\ref{ass: noniid}. Then the upper bound for D-SGD of the same magnitude of convergence when in the neighborhood of solution is $\cO(\sigma^2 + \bar\zeta^2)$~\citep{d2_nonconvex}, which is obviously worse than that for $K$-GT. The additional $\cO(\bar\zeta^2)$ in D-SGD from the data heterogeneity can never be improved if always using the sole stochastic gradient~\citep{dsgd}. 


%

\begin{algorithm}
	\caption{$K$\nobreakdash-GT: Gradient Sum 
Tracking}
\resizebox{\linewidth}{!}{
	\begin{minipage}{1.05\linewidth}
	\label{algo: k-gt}
	\begin{algorithmic}[1]
		\State \textbf{parameters:}
		 $T$: number of communication; $K$: number of local steps; $\eta_c,\ \eta_s$: local, communication stepsize; $\mW$: given topology.
		\State \textbf{Initialize:} $\forall i,j\in[n],\ \xx^{(0)}_i = \xx^{(0)}_j;\ \cc_i^{(0)}= - \nabla F_i(\xx^{(0)}; \xi_i)+\frac{1}{n}\sum_j\nabla F_j(\xx^{(0)}; \xi_j) $\footnote{This initialization in correction $\cc_i^{(0)}$ is required for heterogeneity-independent analysis in theory. We demonstrate later with experiment that simply choosing $\cc_i^{(0)} = \zero$ works well in practice.}.
		\For{client $i\in\{1,\ldots,n\}$ parallel}
			\For{\textbf{communication:} $t\gets0$ to $T-1$}
				\For{\textbf{local step:} $k\gets0$ to $K-1$}
					\State $\xx_i^{(t)+k+1} = \xx_i^{(t)+k} - \eta_c(\nabla F_i(\xx_{i}^{(t)+k}; \xi_i^{(t)+k}) + \cc_i^{(t)})$
				\EndFor
				\State $\zz_i^{(t)} = \tfrac{1}{K\eta_c}(\xx_i^{(t)} - \xx_i^{(t)+K})$
				\State $\cc_i^{(t+1)} = \cc_i^{(t)} - \zz_i^{(t)} + \sum_j w_{ij}\zz_j^{(t)}$ \Comment{update tracking variable}
				\State $\xx_i^{(t+1)} = \sum_jw_{ij}(\xx_j^{(t)}-K\eta_s\eta_c\zz_j^{(t)})$\Comment{update model parameters}
			\EndFor
		\EndFor
	\end{algorithmic}
\end{minipage}}
\end{algorithm}

\subsection{Main theorem: data-independent convergence on non-convex functions}
In this section, we present the convergence rate of $K$-GT. Note that $p$ is the network parameter defined in Assumption~\ref{ass:p}.

\begin{theorem}[{$K$-GT convergence}]
\label{thm: k-gt}
For schemes as in Algorithm~\ref{algo: k-gt} with mixing matrices such as in Assumption \ref{ass:p} and arbitrary error $\epsilon>0$, there exists a constant stepsize $\eta_c=O(\frac{p}{KL})$ and $\eta_s=O(p)$ such that under Assumption \ref{ass:l} and \ref{ass:noise} for $L$\nobreakdash-smooth, (possibly non-convex) functions, it holds $\frac{1}{T+1}\sum_t\mE||\nabla f(\bar\xx^{(t)})||^2\leq \epsilon$ after $$\cO\left(\frac{\sigma^2}{Kn\epsilon^2} +\frac{\sigma}{p^2\sqrt{K} \epsilon^{\frac{3}{2}}}+\frac{1}{p^2 \epsilon}\right)\cdot L$$ communication rounds.
\end{theorem}



\vspace{-5mm}
\section{Discussion}\label{sec: discuss}
In this section, we are going to introduce and compare with other possible ways of introducing local steps to GT that has the similar communication pattern as $K$\nobreakdash-GT.
%

\vspace{-5mm}
\subsection{Other GT alternatives}\label{sec:alternatives}

 \subsubsection{Gradient Tracking with Periodical Communication (Periodical GT).} \label{sec:perodical_gt}
 There is another way to incorporate local steps into above framework~(\ref{eq: dsgt}). Instead of communication via fixed topology $\mW$, the communication graph changes along with time denoted by $\mW^{(t)}$. Note that $\mW^{(t)} = \mI$, which means there is actually no communication. 
 If $\mW^{(t)}$ periodically alternates between $\{\mW,\ \mI\}$, it also reduces communication frequency. The full detail is concluded in Algorithm~\ref{algo: periodical_gt}~(Appendix~\ref{ap: periodical_gt}).
\vspace{-2mm}

\paragraph*{$K$\nobreakdash-GT suffers from less noise than Periodical GT.} It is possible to reformulate local steps of Periodical GT as corrected SGD, same as that for $K$\nobreakdash-GT. But Periodical GT has different update for correction at communication with
\begin{equation}
		\mC^{(t+1)}=\mC^{(t)}\mW + \nabla F(\mX^{(t)+K-1}; \xi^{(t)+K-1})(\mW-\mI)\,.
	\label{eq: tv_gt}
\end{equation} The equivalence of reformulation is proven in Appendix~\ref{ap: reformulation}.

However, if we simply reformulate equation (\ref{eq: k-gt}), we obtain that $K$\nobreakdash-GT uses the average of $K$ stochastic gradient, 
i.e., $$\mC^{(t+1)}=\mC^{(t)}\mW + \tfrac{1}{K}\textstyle\sum_k\nabla F(\mX^{(t)+k}; \xi^{(t)+k})(\mW-\mI)\,,$$ which can reduce stochastic noise by $K$. Periodical GT uses only one stochastic gradient, thus would suffer more from stochastic noise. 

\paragraph*{Using more random samples on stochastic gradient can reduce noise in Periodical GT.}
 It is trivial to reduce the stochastic noise in~(\ref{eq: tv_gt}) if using more random samples $\{\xi^{(t),s}|s = 0,\ldots,K-1\}$ to replace $\nabla F(\mX^{(t)+K-1}; \xi^{(t)+K-1})$ with\begin{equation}\nonumber\nabla F(\mX^{(t)+K-1}; \xi^{(t)+K-1}) = \frac{1}{K}\sum_s\nabla F(\mX^{(t)+K-1}; \xi^{(t)+K-1,s})\,,\end{equation} then the correction $\mC^{(t)}$ has the same level of stochastic noise as $K$\nobreakdash-GT. However, using more sample to calculate SGD requires a lot more extra computation than $K$\nobreakdash-GT. 

\begin{theorem}[Periodical GT convergence]\label{thm: tv_gt}For schemes as in Algorithm~\ref{algo: periodical_gt}~(Appendix~\ref{ap: periodical_gt}) with mixing matrices such as in Assumption \ref{ass:p} and arbitrary error $\epsilon>0$, there exists a constant stepsize $\eta=O(\frac{p^2}{KL})$ such that under Assumption \ref{ass:l} and \ref{ass:noise} for $L$\nobreakdash-smooth, (possibly non-convex) functions, it holds $\frac{1}{T+1}\sum_t\mE||\nabla f(\bar\xx^{(t)})||^2\leq \epsilon$ after $$\cO\left(\frac{\sigma^2}{Kn\epsilon^2} +\frac{\sigma}{p^2 \epsilon^{\frac{3}{2}}}+\frac{1}{p^2 \epsilon}\right)\cdot L
$$ communication rounds.
Conversely, if we consider using the full-batch tracking Algorithm~\ref{algo: periodical_gt_grad}~(Appendix~\ref{ap: periodical_gt}), then the convergence rate can be improved to $$\cO\left(\frac{\sigma^2}{Kn\epsilon^2} +\frac{\sigma}{p^2\sqrt{K}\epsilon^{\frac{3}{2}}}+\frac{1}{p^2 \epsilon}\right)\cdot L \,.$$
\end{theorem}
\noindent Note that the latter result refers to full batch tracking which comes at additional computation cost each communication round (in contrast to $K$-GT).


\subsubsection{Gradient Tracking with Large Batch (Large-batch GT)}
Apart from local training, large-batch training is also popular to achieve acceleration in distributed setting. Similar to Large-batch SGD, we calculate $\mG^{(t)} = \sum_k\nabla F(\mX^{(t)}; \xi^{(t), k})$ in~(\ref{eq: dsgt}) with $K$ \iid random samples, $\{\xi^{(t),k}|k = 0,\ldots,K-1\}$ and make $\mW^{(t)} = \mW$.  It is theoretically workable to improve the asymptotical communication rounds needed to reach the desired accuracy $\epsilon$ from $\cO\big(\frac{\sigma^2}{n\epsilon^2}\big)$~\citep{koloskova2021improved} to $\cO\big(\frac{\sigma^2}{nK\epsilon^2}\big)$, while remains heterogeneity-independent.

We empirically show that Large-batch GT remains heterogeneity-independent and has the same communication performance as $K$\nobreakdash-GT ~(Figure~\ref{fig: simulation_convex},~\ref{fig: large}).  

\vspace{-4mm }
\subsection{Convergence comparison}

We summarized the convergence rate for the related decentralized algorithms  in Table~\ref{tab: convergence}. In order to analyze the convergence for other methods that depend on data heterogeneity, there is an addition assumption to measure data heterogeneity~\cite{scaffold,unified-dsgd}.
\begin{table}[tb]
\caption{The comparison of communication rounds needed to reach target accuracy $\epsilon$ on non-convex functions. Our results on both $K$-GT, Periodical GT improve the rate of D-SGD in terms of heterogeneity parameter $\bar\zeta^2$(defined in Ass. \ref{ass:zeta}) when using local steps, and accelerates the rate of GT. }
	\centering
	\renewcommand{\arraystretch}{1.2}
	\resizebox{0.75\textwidth}{!}{
		\begin{tabular}{l|l|ll}		
			\toprule[1.1pt]
			\multicolumn{1}{l}{\textbf{Local steps}} &\multicolumn{2}{l}{\textbf{Algorithm}} & \textbf{Communication rounds}	\\\midrule[0.8pt]
			$K=1$ &\multicolumn{2}{l}{GT~\citep{koloskova2021improved}} & $\cO\Big(\frac{\sigma^2}{n \epsilon^2} +\frac{\sigma}{p^{\frac{3}{2}} \epsilon^{\frac{3}{2}}}+\frac{1}{p^2 \epsilon}\Big)$ \\\cline{1-4}
			\multirow{4}{*}{$K>1$}&\multicolumn{2}{l}{D\nobreakdash-SGD~\citep{unified-dsgd}} &$\cO\Big(\frac{\sigma^2}{Kn \epsilon^2} +(\frac{\bar\zeta}{p}+\frac{\sigma}{\sqrt{pK}})\frac{1}{\epsilon^{\frac{3}{2}}}+\frac{1}{p \epsilon}\Big)$ \\\cline{2-4}
			&\multicolumn{2}{l}{$K$\nobreakdash-GT {\textbf{[ours]}}}&  $\cO\Big(\frac{\sigma^2}{Kn \epsilon^2} +\frac{\sigma}{p^2\sqrt{K} \epsilon^{\frac{3}{2}}}+\frac{1}{p^2 \epsilon}\Big)$\\\cline{2-4}
			&\multicolumn{2}{l}{Periodical GT {\textbf{[ours]}}}&  $\cO\Big(\frac{\sigma^2}{Kn \epsilon^2} +\frac{\sigma}{p^2 \epsilon^{\frac{3}{2}}}+\frac{1}{p^2 \epsilon}\Big)$\\\cline{2-4}
			&\multicolumn{2}{l}{Periodical GT w/ full gradient {\textbf{[ours]}}}&  $\cO\Big(\frac{\sigma^2}{Kn \epsilon^2} +\frac{\sigma}{p^2\sqrt{K} \epsilon^{\frac{3}{2}}}+\frac{1}{p^2 \epsilon}\Big)$\\
			\bottomrule[1.1pt]
		\end{tabular}
		} 
\label{tab: convergence}		
\end{table}
\vspace{-2mm }
\paragraph*{$K$\nobreakdash-GT achieves acceleration by local steps in high-noise regime.} When $\epsilon$ is sufficiently small, the noise dominates the convergence rate ($\sigma>0$) and it is not affected by graph parameter $p$ for GT, Periodical GT and $K$\nobreakdash-GT. Then after enough transient time, Periodical GT and $K$\nobreakdash-GT with $\cO(\frac{\sigma^2}{nK\epsilon^2})$ achieves linear speedup by $K$ compared to GT with rate $\cO(\frac{\sigma^2}{n\epsilon^2})$.
In addition, the transient time for $K$\nobreakdash-GT also decreases with $\cO(\frac{1}{\sqrt{K}})$ comparing to GT baseline.

GT methods are in general more sensitive to the network parameter than diffusion methods~\citep{primal>primal-dual}, e.g., D\nobreakdash-SGD, in the non-asymptotical regime. In our analysis of $K$\nobreakdash-GT, the dependency on the network parameter $p$ is worse than for vanilla GT. Combining our analysis with the tighter analysis of GT presented in concurrent work~\cite{koloskova2021improved} would be an interesting future direction---however in this work we focused on the aspect of equipping GT with local steps.

\vspace{-2mm }
\paragraph*{The impact of data heterogeneity is removable for $K$\nobreakdash-GT .} 
$K$-GT does not completely solve data heterogeneity in general, and depends on the data heterogeneity at the initial point, which is the same case in GT. It has been proven for GT that in the non-asymptotic regime a weaker dependence on the data heterogeneity at the initial point actually remains~\cite{koloskova2021improved}. However, with a single round of global communication for the initial iterates (e.g. in Alg.~\ref{algo: k-gt}), we can remove the heterogeneity from the complexity estimates for GT, $K$-GT and Periodical GT. Table~\ref{tab: convergence} removes the initialization terms from the rate to simplify the presentation. On the contrary, heterogeneity $\bar\zeta^2$ under no circumstance can be eliminated for D-SGD and slows down its convergence.

\vspace{-2mm }
\paragraph*{Periodical GT suffers more from noise comparing to $K$\nobreakdash-GT.} In asymptotical regime, the transient time for $K$\nobreakdash-GT $\cO(\frac{\sigma}{\sqrt{K}})$ decrease with local steps while Periodical GT $\cO(\sigma)$ does not. But this noise term can be improved as discussed in section~\ref{sec:perodical_gt}.  If we consider a full gradient in equation~(\ref{eq: tv_gt}), Periodical GT performs similar to $K$\nobreakdash-GT at the expense of extra computation.
\section{Experimental results}
\label{sec:experiments}
We evaluate the effectiveness of $K$-GT by comparing it with D-SGD and periodical GT.
\vspace{-3mm}
\subsection{Setting}
We conduct experiments in two settings. 
\begin{enumerate}
	\item \textsc{Synthetic datasets:} We first construct the distributed least squares objective with $f_i(\xx)=\frac{1}{2}||\mA_i\xx-\bb_i||^2$ with fixed Hessian $\mA_i^2=\frac{i^2}{n}\cdot \mI_d$, and sample each $\bb_i\sim\cN(0,\frac{\bar\zeta^2}{i^2}\cdot\mI_d)$ for each client $i\in[n]$, where $\bar\zeta^2$ can control the deviation between local objectives~\citep{unified-dsgd}. Stochastic noise is controlled by adding Gaussian noise with $\sigma^2=1$.  
	\item \textsc{Real-world dataset, mnist~\cite{mnist}}: We test the case that all clients collaboratively train a convolutional neural network~(CNN)\footnote{Here we only consider a very simple network without Batch Norm layers~\citep{batchnorm} for simplicity, since it inherently assumes that the data distribution is uniform across different batches, which is not the case that we are interested in. The detailed network structure is listed in Appendix~\ref{appendix: network}.} on real-world dataset, \textsc{mnist}. In total this dataset contains 60,000 images of size 28×28 and 10 labels. \end{enumerate}
We use a ring topology for both sets of experiments. For simplicity, instead of using the initialization in Alg.~\ref{algo: k-gt}, we initialize $\cc_i=\zero$ for all experiments.
For data partition on \textsc{mnist}, we consider both homogeneous and heterogeneous cases. The homogeneous dataset is first shuffled and then
uniformly partitioned among all the clients. We call this the `random' setting. The heterogeneous datasets is created when each client only has exclusive access to subset of classes. We call this the `sorted' case, and the data variation across clients is maximized at this time. We use $n=5$ and $n=10$ clients and each client has access to one and two classes case accordingly, and the case $n=10$ has more severe heterogeneity condition than the case $n=5$.
\vspace{-4mm }
\paragraph*{Parameter tuning.} For \textsc{Synthetic datasets}, we use the same learning rate $\eta_s$=1 and $\eta$=1e-3. For \textsc{mnist}, we use the best constant learning rate tuned from $\{$0.5, 0.1, 0.05, 0.01, 0.005, 0.001$\}$ for algorithms and batch size 128 on each client. Note that even though our algorithm is purposed with constant learning rate, using more sophisticated and time-varying learning rate scheduler would definitely bring much better performance. 
\vspace{-4mm }
\paragraph*{Comparison.} We mainly illustrate the acceleration and robustness in convergence rate of the $K$\nobreakdash-GT compared to baseline D\nobreakdash-SGD. We also consider the performance of several GT-variants that supports local steps discussed in Section~\ref{sec: discuss}. 

\vspace{-2mm }
\subsection{Numerical results}
\begin{figure}
	\centering
	\includegraphics[width=1\textwidth]{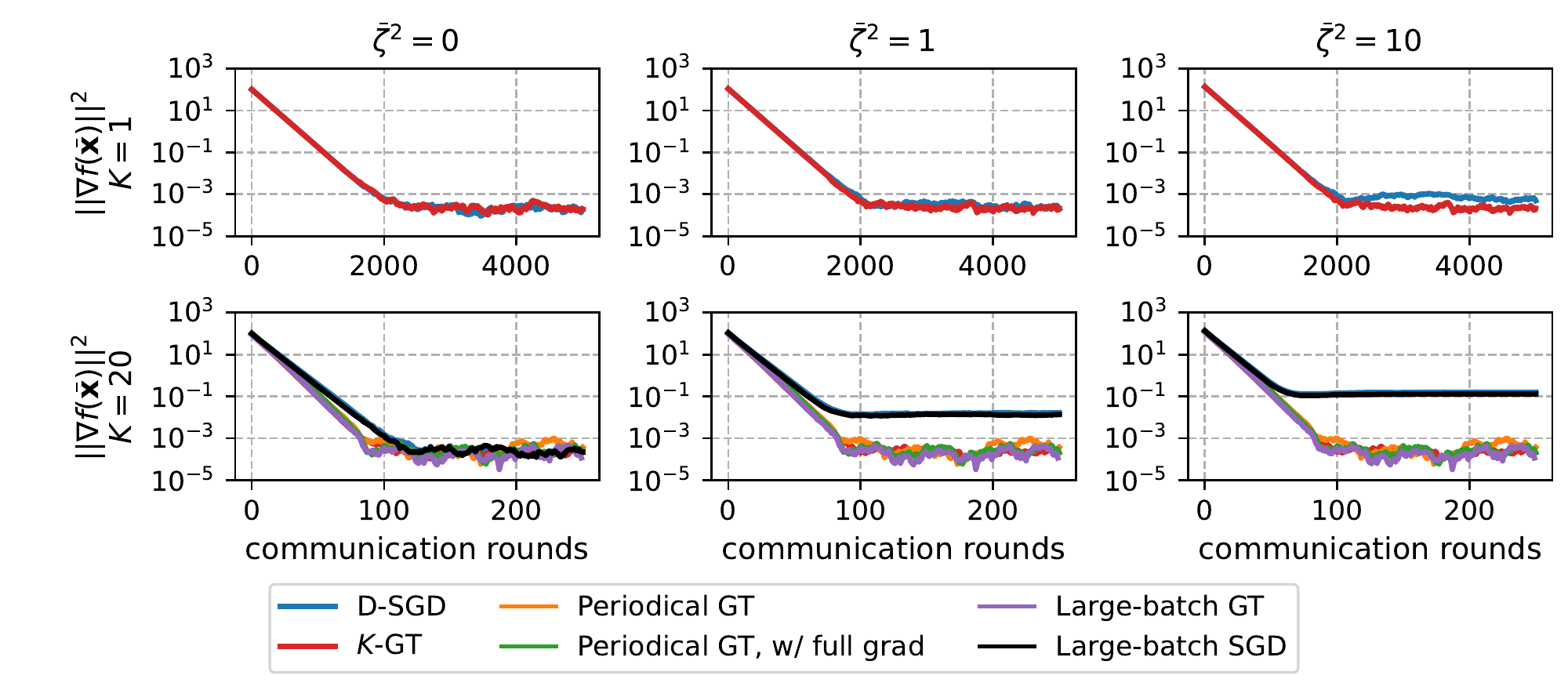}
	\caption{Training \textbf{\textsc{synthetic}} convex functions over ring by 10 clients with noise $\sigma^2 = 1$. In total 5000 communication rounds for $K=1$~(top row) while only 250 rounds for $K=20$~(bottom row). All uses the same learning rate and are averaged by three repetitions. The client\nobreakdash-drift for D\nobreakdash-SGD is even more severe with increasing heterogeneity~(larger $\bar\zeta$) as well as $K$. $K$\nobreakdash-GT, GT, and GT w/ full grad are consistent for different $\bar\zeta^2$ while achieving communication reduction when $K>1$.  }
	\label{fig: simulation_convex}
\end{figure}

\begin{figure}
	\begin{subfigure}{0.5\textwidth}
		\centering
  		\includegraphics[width=1\textwidth]{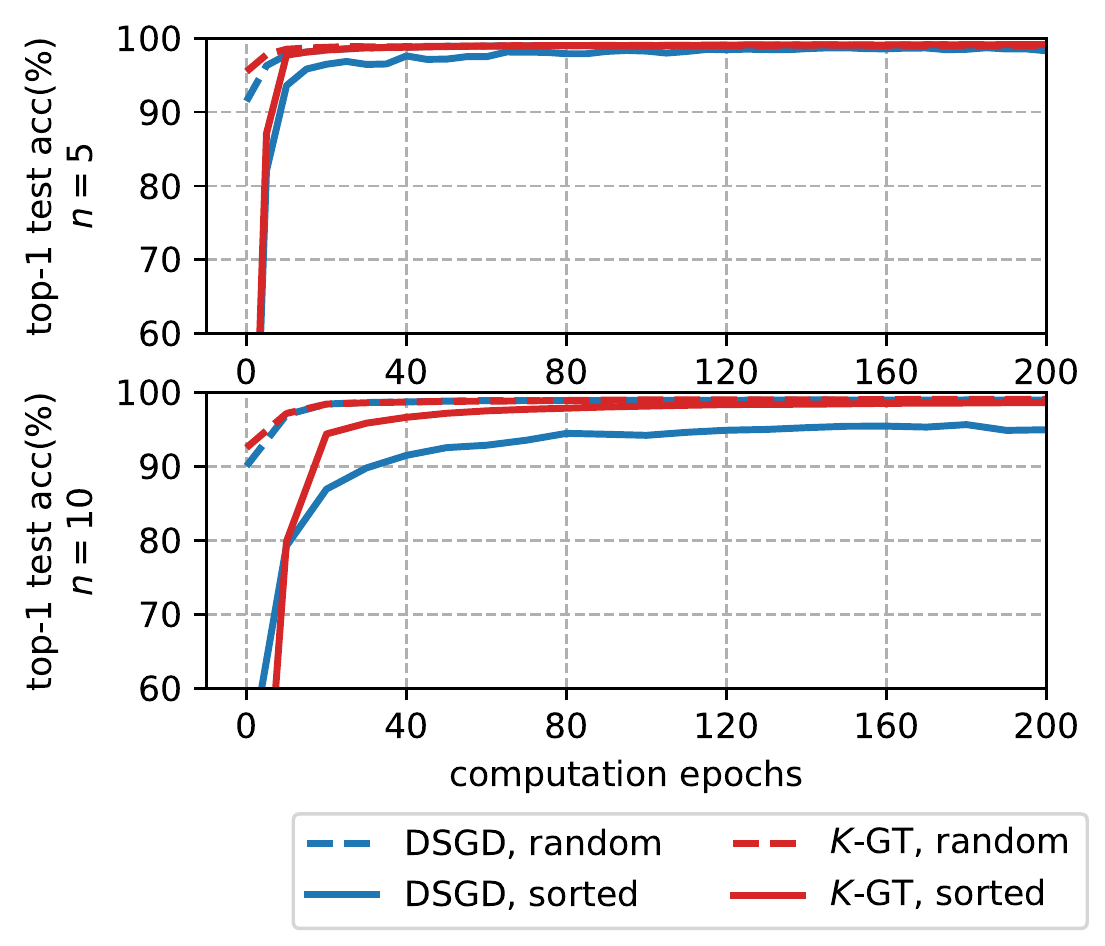}
		\caption{$K=1$, and data is partitioned either random or sorted.}
		\label{fig: baseline}
	\end{subfigure}
	\hfill
	\begin{subfigure}{0.5\textwidth}
		\centering
  		\includegraphics[width=1\textwidth]{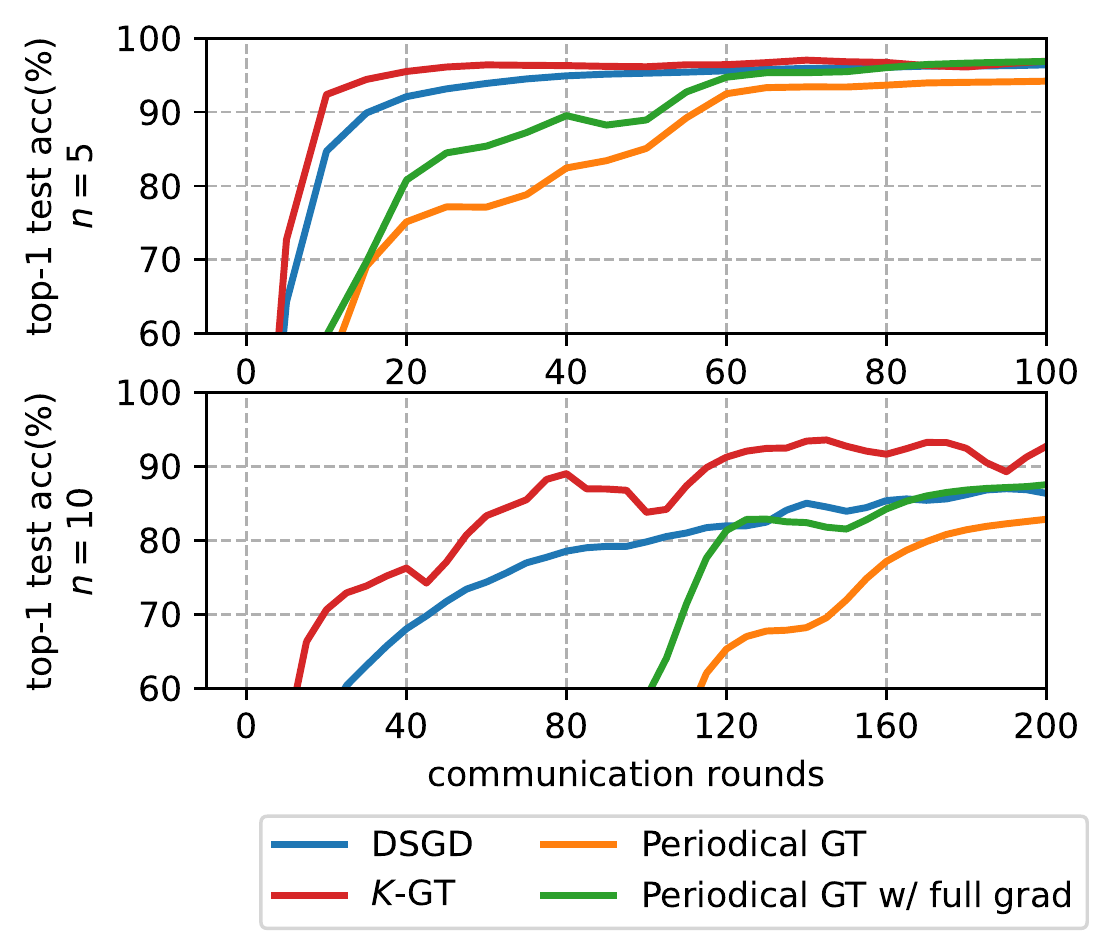}
		\caption{$K=\tfrac{1}{n}$ epoch, and data is partitioned sorted. }
		\label{fig: local_steps}
	\end{subfigure}
	\caption{{Generalization performance} on \textbf{\textsc{mnist}} among D\nobreakdash-SGD~(blue), $K$\nobreakdash-GT(red), GT w/o~(orange) and w/~(green) full gradient for $n=5$~(top row) and $n=10$~(bottom row) clients. The $x$-axis corresponds to (a) the number of pass over overall dataset(epoch), and (b) the number of communication rounds. 
	In (a), when $K=1$, $K$-GT and Periodical GT are identical to GT baseline~(remark \ref{remark: equivalence}). 
	Learning rates are tuned to be the best. Note that 1~epoch of passing over global dataset is equivalent to 470 times computation on SGD when mini-batch sized 128. And for (b), since $nK=470$ is fixed for both $n=5$~(top right) and $n=10$~(top left), the number of local steps between communication rounds is $K_{n=5} = 2K_{n=10}$.}
	\label{fig: mnist}
\end{figure}

\paragraph*{$K$-GT is the most robust against heterogeneity.} In the convex case, client drift only happens for D\nobreakdash-SGD suggested by Figure~\ref{fig: simulation_convex} in which the larger value $\bar\zeta^2\neq 0$ gets, the poorer model quality D\nobreakdash-SGD ends up with. However, $K$\nobreakdash-GT, Periodical GT (w/ and w/o full grad) and Large\nobreakdash-batch GT do not suffer from 'client\nobreakdash-drift' and ultimately reach the consistent level of model quality regardless of increasing of $\bar\zeta$ and $K$ (number of either local steps or random samples). 
In the non-convex case, since it's known the optimality condition and optimization trajectory is more complex than the convex case, generalization performance of all methods in Figure~\ref{fig: mnist} cannot fully recover the baseline performance when data partition is random.  However, $K$-GT could always outperform when data partition is non-\iid and the improvement is more significant when the degree of heterogeneity is increasing from Figure~\ref{fig: local_steps}.

\paragraph*{Local step reduces communication. } From $K=1$ to $K=20$ in Figure~\ref{fig: simulation_convex}, $K$\nobreakdash-GT and other GT alternatives reach the same target after 2000 rounds to only 100 rounds, achieving linear reduction in communication with the help of local steps. However, more local steps makes D\nobreakdash-SGD suffer even more in model quality. At the same time, introducing local steps into the training of non-convex functions would still achieve communication reduction but not by a linear factor of $K$ as in the convex case. Within Figure \ref{fig: local_steps}, we fixe $nK=1$ epoch over the data such that for no matter which $n=5$ or $n=10$ client communicates once after 1 epoch of computation. Compared to $K=1$ in Figure~\ref{fig: simulation_convex}, the acceleration when $K>1$ is still by a huge amount. However, note that introducing more local steps when data partition is heterogeneous would result in more severe quality loss, but $K$-GT still outperforms. 

\paragraph*{Large-batch GT has similar performance to tracking with local steps.} From both convex~(Figure~\ref{fig: simulation_convex}) and non-convex~(Figure~\ref{fig: large}) functions, either $K$-GT or Large-batch GT, after the same number of communication rounds while the simultaneously the same number of computation epochs, reaches the similar level of accuracy. But for D-SGD, training with local steps could be even more stable and generalizes better than the Large-batch, which has been empirically investigated in~\citep{largebatch}.

\begin{figure}
	\centering
  	\includegraphics[width=0.3\textwidth]{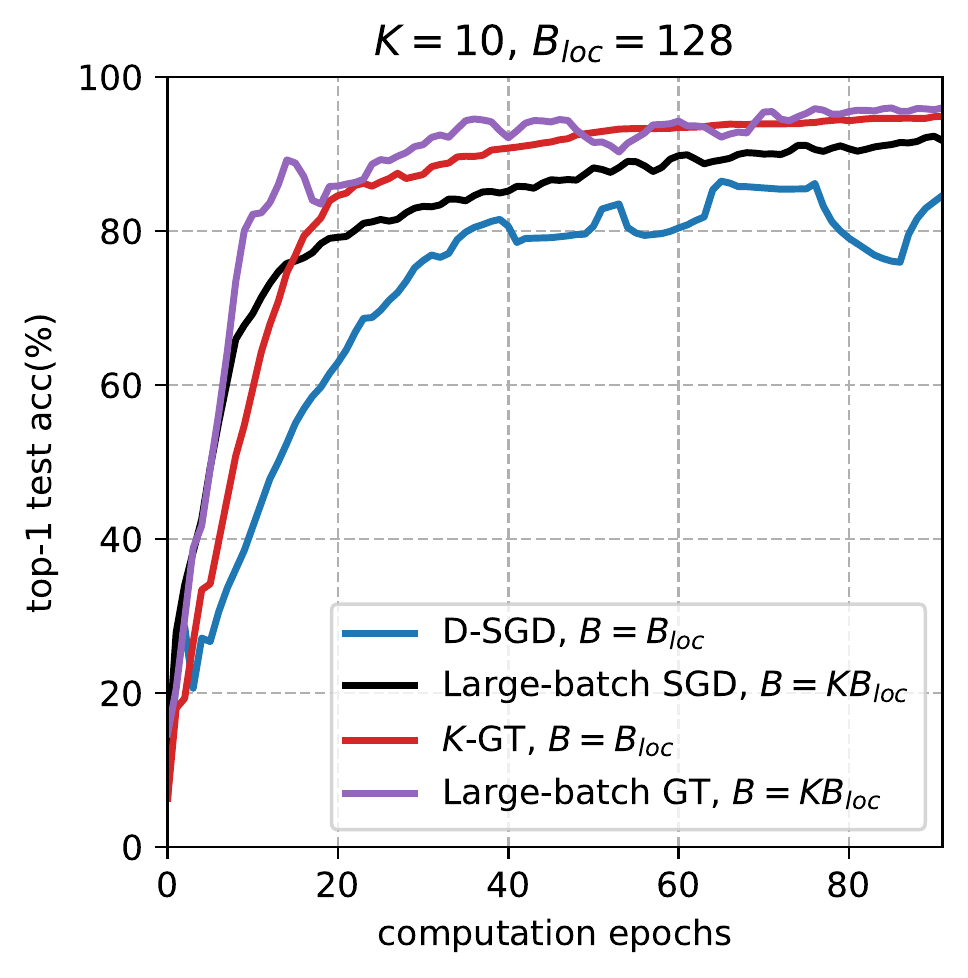}
	\caption{{Generalization performance comparison} on \textbf{\textsc{mnist}} between large-batch training and training with local steps, where large-batch training uses $KB_{loc}$ local batch size and communicates every update, while training with local steps uses $B_{loc}$ local batch size and communicates periodically every $K$ local update.}
	\label{fig: large}
\end{figure}

\vspace{-4mm }
\section{Conclusion}
Decentralized learning is a promising building block for the democratization of Deep Learning.
Especially in Edge AI applications, users' data does not follow a uniform distribution. This requires
robustness of decentralized learning algorithms to data heterogeneity.

We propose a novel decentralized optimization algorithm ($K$\nobreakdash-GT) supports communication efficient local update steps and overcomes data dissimilarity.
The tracking mechanism uses the accumulated gradient sum, akin to momentum, thereby reducing variance across local updates without the need of large batch sizes.
We demonstrated the superiority of $K$\nobreakdash-GT with both convergence guarantees and empirical evaluations. 


\clearpage
\bibliographystyle{tfs}
\bibliography{reference}
\clearpage


\newpage
\appendix


%
\hypersetup{linkcolor=red}
\setcounter{page}{1}

\section{Algorithm}
\setlength{\abovedisplayskip}{2pt}
\setlength{\belowdisplayskip}{0.5pt}
 
 \subsection{Periodical Algorithm}\label{ap: periodical_gt}

\begin{algorithm}[H]
\begin{minipage}{\linewidth}
	\caption{Periodical GT: GT with periodical communication}
	\label{algo: periodical_gt}
	\begin{algorithmic}[1]
		\State \textbf{parameters:}
		\State $T$: number of communication; $K$: number of local steps; $\eta_s, \eta_c$: communication, local stepsize; $\mW$: given topology.
		\State \textbf{Initialize:} $\xx^0_i = \xx^0_j,\ \zz^0_i = \frac{1}{n}\sum_iF_i(\xx^0; \xi_i) = \zz^0_j,\ \forall\ i,j\in[n]$\footnote{This initialization in tracking variable $\zz_i^{(0)}$ is required for heterogeneity-independent analysis in theory. In fact, we show later with experiment that $\zz_i^{(0)} = \nabla F_i(\xx^{0}; \xi)$ works well in practice.}
		\For{node $i\in\{1,...,n\}$ parallel}
		\For{\textbf{communication:} $t\gets0$ to $T-1$}
			\For{\textbf{local steps:} $k\gets0$ to $K-2$}
				\State $\xx_i^{(t)+k+1} = \xx_i^{(t)+k} - \eta_c\zz_i^{(t)+k}$
				\State $\zz_i^{(t)+k+1} = \zz_i^{(t)+k} + \nabla F_i(\xx_i^{(t)+k+1}; \xi_i^{(t)+k+1}) - \nabla F_i(\xx_i^{(t)+k}; \xi_i^{(t)+k})$
			\EndFor
			\State $\xx_j^{(t)+K} = \xx_i^{(t)+K-1} - \eta_c\zz_i^{(t)+K-1} $
			\State $\xx_i^{(t+1)} = \sum_jw_{ij}\Big(\xx_j^{(t)}-\eta_s(\xx_j^{(t)}-\xx_j^{(t)+K})\Big)$
			\State $\zz_i^{(t+1)} = \sum_jw_{ij}\zz_j^{(t)+K-1} + \nabla F_i(\xx_i^{(t+1)}; \xi_i^{(t+1)}) - \nabla F_i(\xx_i^{(t)+K-1}; \xi_i^{(t)+K-1})$
		\EndFor
		\EndFor
	\end{algorithmic}
\end{minipage}
\end{algorithm}
\begin{algorithm}[H]
	\begin{minipage}{\linewidth}
	\caption{Periodical GT with full-batch gradient}
	\label{algo: periodical_gt_grad}
	\begin{algorithmic}[1]
		\State \textbf{parameters:}
		\State $T$: number of communication; $K$: number of local steps; $\eta_s, \eta_c$: communication, local stepsize; $\mW$: given topology.
		\State \textbf{Initialize:} $\xx^0_i = \xx^0_j,\ \cc^0_i = -\nabla F_i(\xx^0; \xi_i) + \frac{1}{n}\sum_j\nabla F_j(\xx^{0}; \xi_j)$\footnote{This initialization in correction $\cc_i^{(0)}$ is required for heterogeneity-independent analysis in theory. In fact, we show later with experiment that $\cc_i^{(0)} = \zero$ works well in practice.}
		\For{node $i\in\{1,...,n\}$ parallel}
		\For{\textbf{communication:} $t\gets0$ to $T-1$}
			\For{\textbf{local steps:} $k\gets0$ to $K-2$}
				\State $\zz_i^{(t)+k} = \nabla F_i(\xx_i^{(t)+k}; \xi_i^{(t)+k}) + \cc_i^{(t)}$
				\State $\xx_i^{(t)+k+1} = \xx_i^{(t)+k} - \eta_c\zz_i^{(t)+k}$
			\EndFor
			\State Compute full gradient on $\xx_i^{(t)+K-1}$, $g_i = \nabla f_i(\xx_i^{(t)+K-1})$.
			\State $\xx_j^{(t)+K} = \xx_i^{(t)+K-1} - \eta_c(g_i + \cc_i^{(t)})$
			\State $\xx_i^{(t+1)} = \sum_jw_{ij}\Big(\xx_j^{(t)+K-1}-\eta_s(\xx_j^{(t)}-\xx_j^{(t)+K})\Big)$
			\State $\cc_i^{(t+1)} = \sum_jw_{ij}\cc_j^{(t)} + \sum_jw_{ij}g_j - g_i$
		\EndFor
		\EndFor
	\end{algorithmic}
\end{minipage}
\end{algorithm}

\section{Proof of proposition}
In this section, we will prove the propositions previously discussed.
\vspace{-3mm}
\subsection{Tracking Property of $K$-GT}\label{ap: gt}
\localgt*
\begin{proof}
The updating schemes of the model for $K$-GT are shown in equation (\ref{eq: k-gt}), then if we define $\mZ^{(t)}=\frac{1}{K\eta_c
}\Big(\mX^{(t)} - \mX^{(t)+K}\Big)$, and $\eta=\eta_s\eta_c$, then with simply reformulating we could derive the set of equations shown above.
\end{proof}

\subsection{Periodical Gradient Tracking reformulation}
\label{sec:periodical_gt}
The periodical GT is actually time-varying GT with skipping communication. That is $\mW^{(t)}=\mW$ in equation (\ref{eq: dsgt}) when mod($t,\ K$)=0, otherwise $\mW^{(t)}=\mI$ no communication and local step. 

Then we adopt the notation for $K$-GT that we denote the model at $k$-th local step after $t$-th communication round as $\mX^{(t)+k}$. And the same principle is applied to tracking variable $\mZ^{(t)+k}$. 
In the following sections, we will first show that Periodical GT can be equivalently reformulated and corrected SGD with constant correction throughout local steps, and then provide the update scheme for both correction and model.
\subsubsection{Corrected SGD}\label{ap: reformulation}

\begin{claim}\label{claim: csgd}
	 The local tracking variable during local steps can be equivalently rewritten as corrected SGD with correction, i.e., $$\mZ^{(t)+k+1} = \nabla F(\mX^{(t)+k+1};\xi^{(t)+k+1})+\underbrace{\mZ^{(t)+k}-\nabla F(\mX^{(t)+k}; \xi^{(t)+k})}_{\mC^{(t)+k}}.$$ And the correction $\mC$ remains unchanged throughout local steps, i.e., $\mC^{(t)+k+1}  = \mC^{(t)+k},\quad \forall k\in\{0,\ ...,\ K-1\}$,  and is updated only  at each time of communication.
\end{claim}
\begin{proof}
We know that local model is updated with $\mZ$ instead of $\nabla F(\mX; \xi)$. We define the deviation of $\mZ$ from the SGD as $\mC$. By contradiction we assume that deviation is different for each local iterate $(t)+k$. That's $\mC^{(t)+k+1} \neq \mC^{(t)+k}.$ Then for each local update, we have 
\begin{equation}\nonumber
	\begin{aligned}
		\mZ^{(t)+k+1} & = \mZ^{(t)+k} + \nabla F(\mX^{(t)+k+1};\xi^{(t)+k+1}) - \nabla F(\mX^{(t)+k}; \xi^{(t)+k})  \\ 
		\mZ^{(t)+k+1}  - \nabla F(\mX^{(t)+k+1};\xi^{(t)+k+1}) & = \mZ^{(t)+k}  - \nabla F(\mX^{(t)+k}; \xi^{(t)+k})  \\ 
		\mC^{(t)+k+1} & = \mC^{(t)+k},\quad \forall k\in\{0,\ ...,\ K-1\}
		\end{aligned}
\end{equation}
which contradicts the assumed fact that $\mC^{(t)+k+1} \neq \mC^{(t)+k}.$ 
\end{proof}
\subsubsection{Updating scheme reformulation}
\begin{proposition}\label{prop: tv_gt}
 If we additionally consider separate step sizes for local steps and communication,  we can equivalently rewrite Periodical GT as follows, 
 \begin{itemize}
\item\textbf{Local steps. }We consider local steps as corrected SGD. The correction $\mC\in\mathbb{R}^{d\times n}$ captures the difference between local update and communication update. For local steps, i.e, $\forall k\in\{0,...,\ K-1\}$, $\mX^{(t)+0}\equiv\mX^{(t)}$,
\begin{equation}\label{eq: k-gt, local}
	\mX^{(t)+k+1}= \mX^{(t)+k} - \eta_c(\nabla F(\mX^{(t)+k}; \xi^{(t)+k}) + \mC^{(t)}),
\end{equation}
where $\mC^{(t)}$ is constant for all local steps.

\item \textbf{Communication. }Then it synchronizes both $\mX$ and $\mC$,
\begin{equation}\label{eq: k-gt, comm}
	\begin{aligned}
		\mX^{(t+1)} & = \Big(\mX^{(t)}-\eta_s(\mX^{(t)} - \mX^{(t)+K})\Big)\mW	\\
		\mC^{(t+1)} & = \mC^{(t)}\mW + \nabla F(\mX^{(t)+K-1}; \xi^{(t)+K-1})(\mW -\mI)	\\
	\end{aligned}
\end{equation}
\end{itemize}
\end{proposition}
\begin{proof}By Claim~\ref{claim: csgd}, the local update is equivalent to corrected SGD. 
Note that different stepsizes $\eta_c$ and $\eta_s$ are used for model update of local step and communication. 

The correction $\mC$ is constant during local steps by Claim~\ref{claim: csgd}, then consider its update during communication.
\begin{equation}\nonumber
	\begin{aligned}
		\mZ^{(t+1)} & = \mZ^{(t)+K-1}\mW +  \nabla F(\mX^{(t+1)};\xi^{(t+1)}) - \nabla F(\mX^{(t)+K-1}; \xi^{(t)+K-1})	\\
		\Leftrightarrow(\nabla F(\mX^{(t+1)}; \xi^{(t+1)}) + \mC^{(t+1)}) & = \Big(\nabla F(\mX^{(t)+K-1}; \xi^{(t)+K-1}) + \mC^{(t)}\Big)\mW+  \nabla F(\mX^{(t+1)};\xi^{(t+1)}) - \nabla F(\mX^{(t)+K-1}; \xi^{(t)+K-1})	\\
		\Leftrightarrow  \mC^{(t+1)} & = \mC^{(t)}\mW + \nabla F(\mX^{(t)+K-1}; \xi^{(t)+K-1})(\mW - \mI)
	\end{aligned}
\end{equation}
\end{proof}
\vspace{-10mm}
\section{Proof of theorem}

\subsection{Technical tools}
In this section, we mainly introduce some analytical tools that help in convergence analysis.
\begin{proposition}[Implications of the smoothness Assumption~\ref{ass: smooth}] Assumption~\ref{ass: smooth} implies $\forall i$ and $\forall \xx,\ \yy\in\mathbb{R}^d$,
\begin{equation}\nonumber
||\nabla f_i(\xx) - \nabla f_i(\yy)||\leq L||\xx-\yy||.
\end{equation}
\end{proposition}
\begin{lemma}
For arbitrary set of $n$ vectors $\{a_i\}_{i=1}^n,\ a_i\in\mathbb{R}^d$, $||\frac{1}{n}\sum_i^na_i||^2\leq \frac{1}{n}\sum_i^n||a_i||^2.$
\end{lemma}
\begin{lemma}
For given two vectors $a,\ b\in\mathbb{R}^d$, $2\langle a,b\rangle\leq \alpha||a||^2+\frac{1}{\alpha}||b||^2,\ \alpha>0,$ which is equivalent to $||a+b||^2 \leq (1+\alpha)||a||^2 + (1+\frac{1}{\alpha})||b||^2$.
\end{lemma}
\begin{remark}
Above inequality also holds for matrix in Frobenius norm.  For $\mA, \mB\in\mathbb{R}^{d\times n}$, $||\mA\mB||_F\leq ||\mA||_F||\mB||_2.$
\end{remark}
\begin{lemma}[Variance upperbound]\label{lemma: variance}If there exist $n$ zero-mean random variables $\{\xi_i\}_{i=1}^n$ that may not be independent of each other, but all have variance smaller than $\sigma^2$, then the variance of sum is upperbounded by $\mE||\sum_i\xi_i||^2\leq n\sigma^2.$
\end{lemma}
\begin{proof}
$\mE||\sum_i\xi_i||^2\leq \mE\Big(n\sum_i||\xi_i||^2\Big)\leq n\sigma^2.$
\end{proof}
\begin{lemma}[{Unrolling recursion~\citep{unified-dsgd}}]\label{lemma: stepsize}For any parameters $r_0\geq0, b\geq0, e\geq0, u\geq0$ there exists constant stepsize $\eta\leq\frac{1}{u}$ such that
	$$\Psi_T:=\frac{r_0}{T+1}\frac{1}{\eta}+b\eta+e\eta^2\leq 2(\frac{br_0}{T+1})^{\frac{1}{2}}+2e^{\frac{1}{3}}(\frac{r_0}{T+1})^{\frac{2}{3}}+\frac{ur_0}{T+1}$$
\end{lemma}

%

\subsubsection*{\textbf{Additional definitions} }
Before proceeding with the proof of the convergence theorem, we need some addition set of definitions of the various errors we track. For simplicity, we define the special matrix $\mJ = \frac{\unit_n\unit_n^T}{n}$ as it could be used to calculate the \textbf{averaged matrix}, $\mX\mJ = \bar\mX=\begin{bmatrix}\bar\xx&\bar\xx&...&\bar\xx\end{bmatrix}.$ 

We define the \textbf{client variance} (or consensus distance) to be how much each node deviates from their averaged model:
$\Xi_t = \frac{1}{n}\sum_i^n\mE||\xx_i^{(t)} - \bar\xx^{(t)}||^2.$

Since we are doing local steps between communication, we define the local progress to be how much each node moves from the globally averaged starting point as \textbf{client-drift}:
\begin{itemize}
	\item at $k$-th local step: $e_{k,t} : = \frac{1}{n}\sum_i^n\mE||\xx_i^{(t)+k} - \bar\xx^{(t)}||^2$
	\item accumulation of local steps: $\cE_{t}  : = \sum_{k=0}^{K-1}e_{k,t} = \sum_{k=0}^{K-1} \frac{1}{n}\sum_i^n\mE||\xx_i^{(t)+k} - \bar\xx^{(t)}||^2$
\end{itemize}

Because we update model with correction, the corrected gradient will be aligned with the direction of the global update instead of the local update. The correction is updated every communication, and remains constant during local steps. We define the \textbf{quality of this correction} to be how much it approximates the true deviation between global update and local update, $\gamma_t = \frac{1}{nL^2}\mE||\mC^{(t)} +\nabla f(\bar\mX^{(t)}) - \nabla f(\bar\mX^{(t)})\mJ||_F^2$, where $\mJ =\frac{1}{n}\unit\unit^{T}$.

\subsection{Convergence analysis} \label{sec:scaffold}

This section we will show the proof of Theorem~\ref{thm: k-gt} and Theorem~\ref{thm: tv_gt}. Since from the previous analysis that $K$-GT and periodical GT are equivalent to corrected SGD for local step, and have similar pattern during communication. We can analyze them within the same prove framework. 

In order to prove the theorems, we first provide the recursion for client-drift, consensus distance and qualify of correction in following sections.


\subsubsection*{\textbf{Bounding the client drift}}\label{sec: client_drift}
We will next consider the progress made within local steps. That's the accumulated model update before next communication.
\begin{lemma}\label{lemma: client_drift}Suppose the local step-size for node $\eta_c\leq\frac{1}{8KL}$, and for arbitrary communication step size $\eta_s\geq 0$, we could bound the drift as
	$$\cE_t \leq 3(K\Xi_t) + 12K^2\eta_c^2L^2(K\gamma_t) + 6K^2{\eta_c}^2(K\mE||\nabla f(\bar\xx^{(t)})||^2) + 3K^2{\eta_c}^2\sigma^2$$
\end{lemma}
\begin{proof}First, observe that $K=1$,
$\cE_t=\frac{1}{n}\mE||\mX^{(t)}-\bar\mX^{(t)}||_F^2,$$$0\leq \frac{2}{n}\mE||\mX^{(t)}-\bar\mX^{(t)}||_F^2+6{\eta_c}^2\mE||\nabla f(\bar\xx^{(t)})||^2 + 3{\eta_c}^2\sigma^2,$$ the inequality will always hold since RHS is always positive. Then the lemma is trivially proven for $K=1$.
\\Then we consider the case for $K\geq 2$, and 	\begin{equation}\nonumber
		\begin{aligned}
			ne_{k,t} &:=\mE||\mX^{(t)+k} - \bar\mX^{(t)}||_F^2 	\\
			& = \mE||\mX^{(t)+k-1} - {\eta_c}\Big(\nabla F(\mX^{(t)+k-1}; \xi^{(t)+k}) + \mC^{(t)}\Big) - \bar\mX^{(t)}||_F^2	\\
			& \leq (1+\frac{1}{K-1})\mE||\mX^{(t)+k-1} - \bar\mX^{(t)}||_F^2+ n{\eta_c}^2\sigma^2 \\&\quad+ K{\eta_c}^2\mE||\nabla f(\mX^{(t)+k-1}) - \nabla f(\bar\mX^{(t)}) + \mC^{(t)} +\nabla f(\bar\mX^{(t)})(\mI - \mJ) + \nabla f(\bar\mX^{(t)})\mJ||_F^2 	\\
			& \leq \underbrace{(1+\frac{1}{K-1}+4K{\eta_c}^2L^2)}_{:=\cC}\mE||\mX^{(t)+k-1} - \bar\mX^{(t)}||_F^2+ 4K{\eta_c}^2L^2n\gamma_t + 2K{\eta_c}^2n\mE||\nabla f(\bar\xx^{(t)})||^2 + n{\eta_c}^2\sigma^2	\\
			& \leq \cC^{k}\mE||\mX^{(t)} - \bar\mX^{(t)}||_F^2 + \sum_{r=0}^{k-1}\cC^{r}\Big(4K{\eta_c}^2L^2n\gamma_t + 2K{\eta_c}^2n\mE||\nabla f(\bar\xx^{(t)})||^2 + n{\eta_c}^2\sigma^2\Big)	\\
		\end{aligned}
	\end{equation}
	If $\eta_c\leq \frac{1}{8KL}$, then $4K(\eta_cL)^2\leq\frac{1}{16K}<\frac{1}{16(K-1)}$. Since $\cC>1$, then $\cC^{k}\leq \cC^{K}\leq  (1+\frac{1}{K-1}+\frac{1}{16(K-1)})^{K}\leq e^{1+\frac{1}{16}} \leq 3$, and $\sum_{r}^{k-1}\cC^{r}\leq K\cC^{K}\leq 3K.$ We could rewrite the bound on client drift at $k^{th}$ local step,
\begin{equation}\label{eq: epsilon_t}
ne_{k,t} \leq 3\Xi_t+ 3K\Big(4K{\eta_c}^2L^2n\gamma_t + 2K{\eta_c}^2n\mE||\nabla f(\bar\xx^{(t)})||^2 + n{\eta_c}^2\sigma^2\Big)
\end{equation}
Clearly, in inequality (\ref{eq: epsilon_t}), the RHS is independent of time step $k\in[0,\ K)$. Then the accumulated progress within local steps $\cE_t$ could be formulated by
\begin{equation}\nonumber
	\begin{aligned}
		\cE_t: =\sum_{k=0}^{K-1}e_{k,t} \leq 3(K\Xi_t) + 12K^2{\eta_c}^2L^2(K\gamma_t) + 6K^2{\eta_c}^2(K\mE||\nabla f(\bar\xx^{(t)})||^2) + 3K^2{\eta_c}^2\sigma^2
	\end{aligned}
\end{equation}
\end{proof}
\vspace{-10mm}
\subsubsection*{\textbf{Consensus distance}}\label{sec: consensus}
We then consider how the consensus distance for communicated model is developed between communications after local training.
\begin{lemma}\label{lemma: consensus} For any effective step-size $\eta=\eta_s\eta_c$,
	we have the descent lemma for $\Xi_t$ as $$\Xi_{t+1}\leq (1-\frac{p}{2})\Xi_t + \frac{6K{\eta}^2L^2}{p}\cE_t + \frac{6K^2{\eta}^2L^2}{p}\gamma_t + K{\eta}^2\sigma^2.$$
\end{lemma}
\begin{proof}

We know that the update between two communication round is as follows, $$\mX^{(t+1)} = \Big(\mX^{(t)}-K\eta\mZ^{(t)}\Big)\mW,$$ where $\mZ^{(t)}= \frac{1}{K}\sum_k\Big(\nabla F(\mX^{(t)+k}; \xi^{(t)+k}) + \mC^{(t)}\Big)$. 
Then consensus distance at time $(t+1)$ can be measured by 
\begin{equation}\nonumber
	\begin{aligned}
		n\Xi_{t+1} & =\mE||\mX^{(t+1)} - \bar\mX^{(t+1)}||_F^2  	\\
		& = \mE||\Big(\mX^{(t)} - {\eta}\sum_{k=0}^{K-1}(\nabla F(\mX^{(t)+k}; \xi^{(t)+k}) + \mC^{(t)})\Big)(\mW -\mJ) ||_F^2	\\
		&\leq(1-p)\mE||\Big(\mX^{(t)} - {\eta}\sum_{k=0}^{K-1}(\nabla f(\mX^{(t)+k}) + \mC^{(t)})\Big)(\mI-\mJ)||_F^2+nK{\eta}^2\sigma^2\\
		& \leq nK{\eta}^2\sigma^2+ (1+\alpha)(1-p)\mE||\mX^{(t)}(\mI-\mJ)||_F^2\\&\quad + (1+\frac{1}{\alpha}){\eta}^2\mE||\sum_{k=0}^{K-1}\nabla f(\mX^{(t)+k})(\mI-\mJ) \pm K\nabla f(\bar\mX^{(t)})(\mI-\mJ)+K\mC^{(t)}||_F^2	\\
		& \underset{\alpha=\frac{p}{2},\frac{1}{p}\leq 1}{\leq} nK{\eta}^2\sigma^2 + (1-\frac{p}{2})\mE||\mX^{(t)} - \bar\mX^{(t)}||_F^2 + \frac{6}{p}\Big(K{\eta}^2L^2||\mI-\mJ||^2\sum_{k=0}^{K-1}\mE||\mX^{(t)+k} - \bar\mX^{(t)}||_F^2 \\&\quad + K^2{\eta}^2\frac{L^2}{L^2}\mE||f(\bar\mX^{(t)})(\mI - \mJ)+\mC^{(t)}||_F^2\Big) 	\\
		& \leq (1-\frac{p}{2})n\Xi_t + \frac{6K{\eta}^2L^2}{p}n\cE_t + \frac{6K^2{\eta}^2L^2}{p}n\gamma_t + nK{\eta}^2\sigma^2
	\end{aligned}
\end{equation}
\end{proof}
\vspace{-5mm}
\subsubsection*{\textbf{Quality measure of correction}}\label{sec: correction}
We now bound the quality measure of correction. The correction is thought to depict the deviation of local and global gradient of the ideally averaged model $\bar\mX$ at the time of communication. That is, quality measure of correction is defined to be $\gamma_t = \frac{1}{nL^2}\mE||\mC^{(t)} +\nabla f(\bar\mX^{(t)}) - \nabla f(\bar\mX^{(t)})\mJ||_F^2$, where $\mJ =\frac{1}{n}\unit\unit^{T}$. 

How to estimate correction, $K$-GT and periodical GT have different options, which is carefully discussed in section \ref{sec:alternatives}.

\begin{lemma}\label{lemma: correction} For any effective step-size $\eta=\eta_s\eta_c\leq\frac{\sqrt{p}}{\sqrt{6}KL}$, we have the descent lemma for $\gamma$ in periodical GT as follow,
$$K\gamma_{t+1} \leq (1-\frac{p}{2})K\gamma_t+\frac{24}{p}(Ke_{K-1, t})+\frac{2}{p}\cE_t +\frac{12K^2\eta}{p}K\eta\mE||\nabla f(\bar\xx^{(t)})||^2+ \frac{2K\sigma^2}{L^2},$$ and if we instead of using the average of local steps for $K$-GT in correction, we have the descent lemma for $\gamma$ as follow,
$$K\gamma_{t+1} \leq (1-\frac{p}{2})K\gamma_t+\frac{30}{p}\cE_t +\frac{12K^2{\eta}^2}{p}K\mE||\nabla f(\bar\xx^{(t)})||^2+ \frac{2\sigma^2}{L^2}.$$ 
\end{lemma}

\begin{proof}
The averaged correction between two consecutive communication round satisfies $$\mC^{(t+1)}\mJ = \mC^{(t)}\mJ+ \frac{1}{K\eta_c}(\mX^{(t)} - \mX^{(t)+K})(\mW-\mI)\mJ  = \mC^{(t)}\mJ$$ for $K$-GT. We assume that the correction is initialized with arbitrary value as long as its globally average always equals to zero, i.e., $\mC^{(t)}\mJ=\mC^{(0)}\mJ = \zero$. Recall the definition of $\gamma_t$,  note that $$ (\mC^{(t)}+ \nabla f(\bar\mX^{(t)}) -\nabla f(\bar\mX^{(t)}\mJ)\mJ = \mC^{(t)} \mJ+ \nabla f(\bar\mX^{(t)})(\mJ - \mJ)=\zero.$$ It's easy to check that Periodical GT has the same property.


Then for $K$-GT, we have the recursion of quality measure can be formulated as follows,
\begin{equation}\nonumber
	\begin{aligned}
		nL^2\gamma_{t+1} &: = \mE||\mC^{(t+1)} + \nabla f(\bar\mX^{(t+1)})(\mI - \mJ)||_F^2	\\
		& = \mE||\mC^{(t)}\mW + \frac{1}{K}\sum_{k=0}^{K-1}\nabla F(\mX^{(t)+k}; \xi^{(t)+k})(\mW - \mI) + \nabla f(\bar\mX^{(t+1)})(\mI - \mJ) ||_F^2	\\
		& = \mE||\Big(\mC^{(t)} + \nabla f(\bar\mX^{(t)})(\mI - \mJ)\Big)\mW \\&\quad+ \Big(\frac{1}{K}\sum_{k=0}^{K-1}\nabla f(\mX^{(t)+k}) - \nabla f(\bar\mX^{(t)})\Big)(\mW - \mI)\\&\quad +\Big( \nabla f(\bar\mX^{(t+1)}) - \nabla f(\bar\mX^{(t)})\Big)(\mI - \mJ) ||_F^2 + \frac{n\sigma^2}{K}		\\
		& \leq (1+\alpha)(1-p)nL^2\gamma_t + 2(1+\frac{1}{\alpha})\Big(||\mW-\mI||^2\frac{1}{K}\sum_{k=0}^{K-1}\mE||\nabla f(\mX^{(t)+k}) - \nabla f(\bar\mX^{(t)})||_F^2 \\&\quad+ ||\mI - \mJ||^2\mE||\nabla f(\bar\mX^{(t+1)}) - \nabla f(\bar\mX^{(t)})||_F^2\Big) + \frac{n\sigma^2}{K} \quad(\mbox{due to }||\mW-\mI||\leq 2,\ ||\mI-\mJ||\leq 1)\\
		&\underset{\alpha=\frac{p}{2},\ \frac{1}{p}\leq 1}{ \leq} (1-\frac{p}{2})nL^2\gamma_t + \frac{6}{p}\Big(\frac{4}{K}\sum_{k=0}^{K-1}L^2\mE||\mX^{(t)+k}-\bar\mX^{(t)}||_F^2 + nL^2\mE||\bar\xx^{(t+1)} - \bar\xx^{(t)}||^2\Big)+ \frac{n\sigma^2}{K}	 \\
		& \leq (1-\frac{p}{2})n\gamma_t + \frac{6L^2}{pK}\Big(4n\cE_t + 2K^2{\eta}^2L^2n\cE_t+ 2K^2{\eta}^2n(K\mE||\nabla f(\bar\xx^{(t)})||^2) + K^2{\eta}^2\sigma^2\Big)+ \frac{n\sigma^2}{K}	\\
		\end{aligned}
\end{equation}

Periodical GT uses $\nabla F(\mX^{(t)+K-1}; \xi^{(t)+K-1})$ to replace $\frac{1}{K}\sum_{k=0}^{K-1}\nabla F(\mX^{(t)+k}; \xi^{(t)+k})$ in correction update. With almost identical analysis, we could get a very similar equality. In addition, if we replace stochastic gradient with full gradient, i.e, $\nabla f(\mX^{(t)+K-1})$, in periodical GT, which will improve the noise term.


Further, if $\eta=\eta_s\eta_c\leq \frac{\sqrt{p}}{\sqrt{6}KL}$, then $\frac{6K^2\eta^2L^2}{p}\leq 1$ which completes the proof.

\end{proof}
\begin{remark}
Shown from results above, the quantizations of $\gamma_t$ for periodical GT and $K$-GT only differ in the coefficient of stochastic noise. And using full-batch gradient can improve Periodical GT in stochastic noise to the same level as that of $K$-GT.
\end{remark}

\subsubsection*{\textbf{Progress between communications}} \label{sec: comm_progress}
We study how the progress between communication rounds could be bounded.

\begin{lemma}\label{lemma: comm_progress} We could bound the averaged progress between communication in any round $t\geq 0$,  and any $\eta=\eta_s\eta_c\geq 0$ as follows,
	$$\mE||\bar\xx^{(t+1)} - \bar\xx^{(t)}||^2\leq 2K{\eta}^2L^2\cE_t + 2K^2{\eta}^2\mE||\nabla f(\bar\xx^{(t)})||^2+ \frac{(K{\eta})^2\sigma^2}{nK} \,.$$
\end{lemma}
\begin{proof}
From previous analysis, we guarantee $\frac{1}{n}\sum_i\cc_i^{(t)} = \zero$. Then the averaged progress between communication could be rewritten as
	\begin{equation}\nonumber
		\begin{aligned}
			\mE||\bar\xx^{(t+1)}-\bar\xx^{(t)}||^2 & = {\eta}^2\mE||\frac{1}{n}\sum_{i,k}\nabla F_i(\xx_i^{(t)+k}; \xi^{(t)+k}) + \frac{K}{n}\sum_i\cc_i^{(t)}||^2		\\
			& \leq \frac{K{\eta}^2}{n}\sum_{i,k}2\mE||\nabla f_i(\xx_i^{(t)+k}) - \nabla f_i(\bar\xx^{(t)}) ||^2 + 2K^2{\eta}^2\mE||\nabla f(\bar\xx^{(t)})||^2+ \frac{K{\eta}^2\sigma^2}{n}	\\
			& \leq \frac{2K{\eta}^2L^2}{n}\sum_{i,k}\mE||\xx_i^{(t)+k} - \bar\xx^{(t)}||^2 + 2K^2{\eta}^2\mE||\nabla f(\bar\xx^{(t)})||^2+ \frac{K{\eta}^2\sigma^2}{n}	\\
		\end{aligned}
	\end{equation}
	In the first inequality, note that the $K$ random variable $
\{\xi^{(t)+k}\}_{k=0}^{K-1}$ when conditioned on communication $(t)$ may not be independent of each other but each has variance smaller than $\sigma^2$ due to Assumption~\ref{ass:noise}, and we can apply Lemma \ref{lemma: variance}. Then the following inequalities are from the repeated application of triangle inequality.
\end{proof}
\vspace{-6mm}

\subsubsection*{\textbf{Descent lemma for non-convex case}}\label{sec: descent}
\begin{lemma}\label{lemma: descent}When function $f$ is $L$-smooth, the averages $\bar\xx^{(t)}$ of the iterates of Algorithm \ref{algo: k-gt} and Algorithm \ref{algo: periodical_gt} with the constant stepsize $\eta_c < \frac{1}{4\eta_sKL}$, satisfy
$$\mE f(\bar\xx^{(t+1)}) - \mE f(\bar\xx^{(t)})\leq-\frac{K{\eta}}{4}\mE||\nabla f(\bar\xx^{(t)})||^2 + \eta L^2\cE_t+\frac{K{\eta}^2L}{2n}\sigma^2$$
\end{lemma}
\begin{proof}Because the local functions $\{f_i(\xx)\}$ are $L$-smooth according to Assumption~\ref{ass: smooth}, it's trivial to conclude that the global function $f(\xx)$ is also $L$-smooth.
\begin{equation}\nonumber
	\begin{aligned}
		\mE f(\bar\xx^{(t+1)}) & = \mE f\Big(\bar\xx^{(t)}-\frac{\eta}{n}\sum_{i,k}(\nabla F_i(\xx^{(t)+k}_i;\xi_i^{(t)+k})+\cc_i^{(t)})\Big)	\\
			&\leq   \mE f(\bar\xx^{(t)}) + \underbrace{\mE\Big\langle\nabla f(\bar\xx^{(t+1)}),-\frac{\eta}{n}\sum_{i,k}(\nabla F_i(\xx_i^{(t)+k};\xi_i^{(t)+k})+\cc_{i}^{(t)})\Big\rangle}_{:=U} + \frac{L}{2}\mE||\bar\xx^{(t+1)}-\bar\xx^{(t)}||^2
	\end{aligned}
\end{equation}
From our previous analysis, we know $\frac{1}{n}\sum_i\cc_i^{(t)}=0,\ \forall t\geq 0$ for$K$-GT and Periodical GT~(if with initialization indicated in purposed algorithm)
\begin{equation}\nonumber
	\begin{aligned}
		U:&=\mE\Big\langle\nabla f(\bar\xx^{(t+1)}), -\frac{\eta}{n}\sum_{i,k}(\nabla F_i(\xx_i^{(t)+k};\xi_i^{(t)+k})+\cc_{i}^{(t)})\Big\rangle	\\
		& = \mE\Big\langle\nabla f(\bar\xx^{(t)}),\ -\frac{\eta}{n}\sum_{i,k}\mE_{\xi_i^{(t)+k}}\nabla F_i(\xx_i^{(t)+k}; \xi_i^{(t)+k})\Big\rangle	\\
		&=-K{\eta}\mE\Big\langle\nabla f(\bar\xx^{(t)}), \frac{1}{nK}\sum_{i,k}\nabla f_i(\xx_i^{(t)+k})-\nabla f(\bar\xx^{(t)}) + f(\bar\xx^{(t)})\Big\rangle	\\
		&= -K{\eta}\mE||\nabla f(\bar\xx^{(t)})||^2+ \frac{1}{nK}\sum_{i,k}K{\eta}\mE\Big\langle\nabla f(\bar\xx^{(t)}), \Big(\nabla f_i(\xx_i^{(t)+k})-\nabla f_i(\bar\xx^{(t)})\Big)\Big\rangle	\\
		& \leq - \frac{K\eta}{2}\mE||\nabla f(\bar\xx^{(t)})||^2 + \frac{K\eta}{2nK}\sum_{i,k}\mE||\nabla f_i(\xx^{(t)+k}_i)-\nabla f_i(\bar\xx^{(t)})||^2	\\
		& \leq - \frac{K\eta}{2}\mE||\nabla f(\bar\xx^{(t)})||^2 + \frac{KL^2\eta}{2nK}\sum_{i,k}\mE||\xx_i^{(t)+k}-\bar\xx^{(t)}||^2
	\end{aligned}
\end{equation}
Then also plug in the Lemma \ref{lemma: comm_progress} for $\mE||\bar\xx^{(t+1)}-\bar\xx^{(t)}||^2$, we have
\begin{equation}\nonumber
	\begin{aligned}
		\mE f(\bar\xx^{(t+1)})	
		& \leq  \mE f(\bar\xx^{(t)})+(- \frac{K{\eta}}{2} + K^2{\eta}^2L)\mE||\nabla f(\bar\xx^{(t)})||_F^2 + (\frac{{\eta} L^2}{2}+K{\eta}^2L^3)\cE_t+\frac{K{\eta}^2L}{2n}\sigma^2	\\
	\end{aligned}
\end{equation}
Then the choice  ${\eta}\leq\frac{1}{4KL}$ completes the proof.
\end{proof}

\subsubsection*{\textbf{Main recursion}}\label{sec: recursion}
We first construct a potential function $\cH_t = \mE f(\bar\xx^{(t)}) - \mE f(\xx^{(\star)}) + A\frac{(K{\eta_c})^3L^4}{p\eta_s^2}\gamma_{t}+\frac{B}{6v^2}\frac{K{\eta_c} L^2}{p}\Xi_{t}$ where constants $A$, $B$ and $v$ can be obtained through the following lemma. 
\begin{lemma}[Recursion for $K$-GT]\label{lemma: potential}
For any effective stepsize of Algorithm~\ref{algo: k-gt} satisfying $\eta_s=\tilde O(\frac{p}{KL})$ and $\eta_c=\tilde O(p)$, there exists constants $A,\ B,\ v$ satisfying $D>0$ and $D_5 \geq 0$. Then we have the recursion 
 $$\cH_{t+1}-\cH_t\leq -DK\eta\mE||\nabla f(\bar\xx^{(t)})||^2+\frac{D_5L^2}{pK}(K\eta)^3\sigma^2 + \frac{L}{2nK}(K\eta)^2\sigma^2.$$
\end{lemma}

\begin{proof}
First, from previous bound on those error term $\gamma_t$, $\Xi_t$ and $\cE_t$, 
we could bound the difference between $\cH_{t+1}$ and $\cH_t$ for $K$-GT, while we also plug in with $C>0$ the $$0\leq-C\eta_c L^2\cE_t + 3C(K\eta_c L^2\Xi_{t}) + 12C(K\eta_c)^3L^4\gamma_t + C\frac{6(K{\eta_c})^2L^2}{\eta_s}K\eta\mE||\nabla f(\bar\xx^{(t)})||^2 + 3C(K{\eta_c})^3\frac{L^2\sigma^2}{K}.$$ Then we have the inequality recursion for $K$-GT as follows, 
\begin{equation}\nonumber
	\begin{aligned}
		\cH_{t+1}-\cH_t&\leq \underbrace{\Big(-\frac{A}{2} + B\frac{\eta_s^2}{v^2p^2} + 12C\Big)}_{\leq D_1}(K{\eta_c})^3L^4\gamma_t	\\
		&\quad+\underbrace{\Big( - \frac{B}{12v^2} + 3C\Big)}_{\leq D_2}K{\eta_c} L^2\Xi_{t} 	\\
		&\quad+ \underbrace{\Big(A\frac{30(K{\eta_c} L)^2}{p^2K} +B\frac{K^2{\eta}^2L^2}{v^2p^2}-C+\eta_s\Big)}_{\leq D_3}{\eta_c} L^2\cE_t	\\
		&\quad + \underbrace{\Big(-\frac{1}{4} + A\frac{12\eta_s(K{\eta_c})^4L^4}{p} +C\frac{6K^2{\eta_c}^2L^2}{\eta_s}\Big)}_{\leq D_4}K{\eta}\mE||\nabla f(\bar x^{(t)})||^2	\\
		&\quad+\underbrace{\Big(A\frac{2L^2}{pK} + B\frac{\eta_s^2L^2}{6v^2pK} + C\frac{3L^2}{K}\Big)}_{\leq\frac{D_5L^2}{pK}}(K{\eta_c})^3\sigma^2+\frac{L}{2nK}(K{\eta})^2\sigma^2
	\end{aligned}
\end{equation}
As long as ${\eta_c}\leq\frac{p}{96vKL},\ \eta_s= v\cdot p\rightarrow \eta_c\eta_s\leq\eta=\frac{p^2}{96KL}$ and $A=72v^3p+48vp,\ B=36v^3p,\ C=vp$ there exists constant $v>1$ that makes $D_1,\ D_2,\ D_3\leq 0$, $D_4<0$. And $D_4\leq -D<0$, and $D_5\geq 0$, which completes the proof.
\end{proof} 

\begin{lemma}[Recursion for Periodical GT]For any effective stepsize of Algorithm~\ref{algo: k-gt} satisfying $\eta_s=\tilde O(\frac{p}{KL})$ and $\eta_c=\tilde O(p)$, there exists constants $A,\ B,\ v$ satisfying $D>0$ and $D_5 \geq 0$. Then we have the recursion 
 $$\cH_{t+1}-\cH_t\leq -DK\eta\mE||\nabla f(\bar\xx^{(t)})||^2+\frac{D_5L^2}{p}(K\eta)^3\sigma^2 + \frac{L}{2nK}(K\eta)^2\sigma^2.$$
\end{lemma}
\begin{proof}
The sets of inequality for Periodical GT only differs in stochastic noise compared to $K$-GT. Then applied with the same principle as that for $K$-GT, we get its recursion of potential function as follows
\begin{equation}\nonumber
	\begin{aligned}
		\cH_{t+1}-\cH_t&\leq  D_1(K{\eta_c})^3L^4\gamma_t + D_2K{\eta_c} L^2\Xi_{t} + D_3{\eta_c} L^2\cE_t + D_4K{\eta}\mE||\nabla f(\bar x^{(t)})||^2	\\
		&\quad+\underbrace{\Big(A\frac{2L^2}{p} + B\frac{\eta_s^2L^2}{6v^2pK} + C\frac{3L^2}{K}\Big)}_{\leq\frac{D_5L^2}{p}}(K{\eta_c})^3\sigma^2+\frac{L}{2nK}(K{\eta})^2\sigma^2-\frac{C}{2}Ke_{K-1,t}
	\end{aligned}
\end{equation}
The rest of the analysis could refer to Lemma~\ref{lemma: potential}.
\end{proof}
\begin{remark} Using full gradient to improve Periodical GT has the same recursion as $K$-GT.
\end{remark}

\subsubsection*{\textbf{Solve the main recursion}}
Take $K$-GT as an example. Consider the telescope sum of the potential function, we can derive
\begin{equation}\nonumber
	\begin{aligned}
		\frac{1}{T+1}\sum_{t=0}^{T}\Big(\cH_{t+1}-\cH_t\Big) & = \frac{1}{T+1}\Big(\cH_{T+1}-\cH_{0} \Big)	\\
		&	 \underset{\eta=\eta_s\eta_c}{\leq} -DK{\eta}\frac{1}{T+1}\sum_{t=0}^{T}\mE||\nabla f(\bar\xx^{(t)})||^2 +  \frac{D_5L^2}{pK\eta_s^3}(K\eta)^3\sigma^2+	\frac{L}{2nK}(K{\eta})\sigma^2	\\
		\underset{\eta_s=v\cdot p}{\Rightarrow}\frac{1}{T+1}\sum_{t=0}^{T}\mE||\nabla f(\bar\xx^{(t)})||^2 & \leq \frac{\cH_{0}-\cH_{T+1}}{(T+1)D}\frac{1}{K{\eta}} +\frac{L\sigma^2}{2nKD}(K{\eta})+\frac{D_5L^2\sigma^2}{v^3p^4KD}(K{\eta})^2	\\
	\end{aligned}
\end{equation}

W.l.o.g we consider that $f(\xx)$ is non-negative. 
Then we could neglect the effect of $-\cH_{T+1}$.
\begin{lemma} There exists constant stepsize such that $$\frac{1}{T+1}\sum_{t=0}^{T}\mE||\nabla f(\bar\xx^{(t)})||^2 =\cO\Big(\sqrt{\frac{\sigma^2L\cH_0}{nKT}}+(\frac{\sigma L\cH_0}{p^2KT})^{\frac{2}{3}}+\frac{L\cH_0}{p^2T}\Big)
.$$
\end{lemma}
\begin{proof}
The non-negative sequences $\{\cH_t\}_{t=0}^{T+1}$ and $\{\mE||\nabla f(\bar\xx^{t})||\}_{t=0}^{T}$ with positive coefficients before both $K\eta$ and $(K\eta)^2$ satisfy the condition in Lemma~\ref{lemma: stepsize}. Then we tune the stepsize using Lemma~\ref{lemma: stepsize}. Then the average of accumulation of gradient could be upper-bounded by
\begin{equation}\nonumber
	\begin{aligned}
	\Rightarrow\frac{1}{T+1}\sum_{t=0}^{T}\mE||\nabla f(\bar\xx^{(t)})||^2 & \leq\cO\Big( 2(\frac{\frac{L\sigma^2}{2nK}\cH_{0}}{T+1})^{\frac{1}{2}}+2(\frac{L^2\sigma^2}{p^4K})^{\frac{1}{3}}(\frac{\cH_0(\xx)}{T+1})^{\frac{2}{3}}+\frac{\cH_{0}(\xx)}{K{\eta}_{max}(T+1)}\Big)	\\
	& = \cO\Big(\sqrt{\frac{\sigma^2L\cH_0}{nKT}}+(\frac{\sigma L\cH_0}{p^2KT})^{\frac{2}{3}}+\frac{L\cH_0}{p^2T}\Big)
	\end{aligned}
\end{equation}
\end{proof}
Then the convergence rate depends on the initial values of potential function $\cH_0$. By the definition of potential function in Lemma~\ref{lemma: potential}, $\cH_0$ is the combination of initial value for $f(\xx^{0})$, $\mE||\mX^{(0)}-\bar\mX^{(0)}||_F^2$ and $\mE||\mC^{(0)} + \nabla f(\bar\mX^{(0)})(-\mJ+\mI)||_F^2$.  

We assume that every node is guaranteed to be initialized with the same model $\xx^{(0)} = \xx_i^{(0)},\ \forall i\in[n]$. Then we could easily get $\mE||\mX^{(0)}-\bar\mX^{(0)}||_F^2=0$.  And if we initial the correction term with $\cc_i^{(0)} = -\nabla f_i(\xx^{(0)}) + \frac{1}{n}\sum_i\nabla f_i(\xx^{(0)})$, then $\mE||\mC^{(0)} + \nabla f(\mX^{(0)})(-\mJ+\mI)||_F^2 = 0$.
\begin{equation}
	\begin{aligned}
		\cH_0 &= f(\xx^0)-f(\xx^{\star}) + A\frac{(K{\eta_c})^3L^4}{p\eta_s^2}\gamma_{0}+\frac{B}{6v^2}\frac{K{\eta_c} L^2}{p}\Xi_{0}	\\
		& = f(\xx^0)-f(\xx^{\star}) := F_0
	\end{aligned}
\end{equation}
And then for arbitrary accuracy error $\epsilon>0$, the communication rounds needed to reach the target accuracy is upperbounded by $$T\leq\cO\Big(\frac{\sigma^2}{nK}\frac{1}{\epsilon^2}+\frac{\sigma}{p^2\sqrt{K}}\frac{1}{\epsilon^{\frac{3}{2}}}+\frac{1}{p^2}\frac{1}{\epsilon}\Big)\cdot LF_0,$$ which concludes the proof of Theorem~\ref{thm: k-gt} for $K$-GT. The proof of Theorem~\ref{thm: tv_gt} for Periodical GT can also be easily derived with the same principle.

\onecolumn
\section{Experimental details}\label{appendix: network}
\subsection{Visualization of benchmark datasets}
\begin{figure}[H]
	\begin{subfigure}[t]{0.15\textwidth}
  		\vspace{12pt}
		\includegraphics[width=1\textwidth]{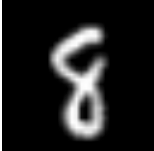}
		\caption{}
	\end{subfigure}
	\hfill
	\begin{subfigure}[t]{0.85\textwidth}
  		\vspace{0pt}
		\includegraphics[width=1\textwidth]{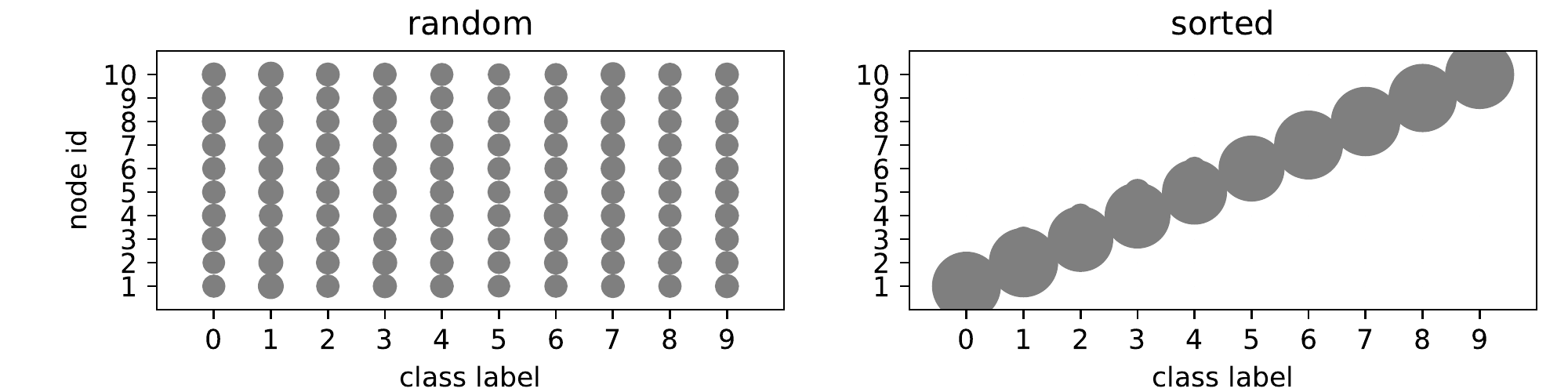}
		\caption{}
	\end{subfigure}
	\caption{Data visualization. (a) Example from \textsc{mnist} dataset. (b) Data partition on each node in the random and the sorted case when there are $n=10$ distributed nodes. The dot size indicates the number of samples per class allocated to each node. }
\end{figure}

We show an image example from \textsc{mnist} datasets and how data of different labels is partitioned in random and sorted case. It obviously presents in the random case, data of different labels are randomly and evenly partitioned among nodes, but in the sorted case, each node only contains images of 1 label and the labels obtained by each node is non-overlapping.
\subsection{Model structure}
For our non-convex experiment, we use a 4-layer Convolutional Neural
Network (CNN) and its details are listed in Table~\ref{tab: network}.
\begin{table}[H]
\caption{Model architecture of the benchmark experiment. For convolutional layer (Conv2D), we
list parameters with sequence of input and output dimension, kernal size, stride. For
max pooling layer (MaxPool2D), we list kernal and stride. For fully connected layer (FC), we list
input and output dimension. For drop out (Dropout), we list the parameter of probability.}
	\centering
	\renewcommand{\arraystretch}{1.2}
	\resizebox{0.75\textwidth}{!}{
		\begin{tabular}{l|l}		
			\toprule[1.1pt]
			layer & details			\\\hline
			1 & Conv2D(1, 10, 5, 1), MaxPool2D(2), ReLU	\\
			2 & Conv2D(10, 10, 5, 1), Dropout2D(0.5), MaxPool2D(2), ReLU	\\
			3 & FC(320, 50), ReLU	\\
			4 & FC(50, 10)	\\
			\bottomrule[1.1pt]
		\end{tabular}}
\label{tab: network}		
\end{table}

\end{document}